\documentclass[reqno, a4paper, 10pt]{amsart}
\usepackage{amssymb,url, color, mathrsfs, bbm, multirow, makecell, array, caption}
\usepackage[colorlinks=true, bookmarks=true, pdfstartview=FitH, pagebackref=true, linktocpage=true]{hyperref}
\usepackage[short,nodayofweek]{datetime}
\usepackage[pagewise,running,mathlines,displaymath, switch]{lineno}
\usepackage{times}
\usepackage{indentfirst, tabularx}
\usepackage[table]{xcolor}
\usepackage{float, hhline, colortbl}
\usepackage{graphicx}
\usepackage{longtable}

\usepackage[framemethod=TikZ]{mdframed}
 \surroundwithmdframed[
  topline=false,
  rightline=false,
  bottomline=false,
  leftmargin=\parindent,
  skipabove=\medskipamount,
  skipbelow=\medskipamount
]{proof}

\newcolumntype{C}[1]{>{\centering\arraybackslash }b{#1}}

\hypersetup{
pdftitle={Existence and non-existence results for the higher order Hardy--H\'enon equation},
pdfauthor={Quoc Anh Ngo and Dong Ye},
colorlinks = true,
linkcolor = magenta,
citecolor = blue,
}

\theoremstyle{plain}
\numberwithin{equation}{section}
\newtheorem{theorem}{Theorem}[section]
\newtheorem{lemma}[theorem]{Lemma}
\newtheorem{corollary}[theorem]{Corollary}
\newtheorem{proposition}[theorem]{Proposition}

\theoremstyle{remark}
\newtheorem{remark}[theorem]{Remark}

\DeclareMathOperator{\G}{{\mathbb G}}
\DeclareMathOperator{\R}{{\mathbf R}}

\newcommand{\F}{\mathscr F}
\DeclareMathOperator{\strokedint}{\int\mkern-17.8mu-\mkern-0.0mu-\mkern-5.0mu}
\newcommand{\fint}{\strokedint}

\definecolor{brown}{rgb}{0.5,0,0}
\definecolor{backgroundcolor}{rgb}{0.98, 0.92, 0.73}

\newcommand{\ps}{p_{\mathsf S}}
\newcommand{\pc}{p_{\mathsf{C}}}

\newcommand{\psnma}{p_{\mathsf S}(m, \sigma)}
\newcommand{\pcnma}{p_{\mathsf C}(m, \sigma)}

\newcommand{\psnmz}{p_{\mathsf S}(m,0)}
\newcommand{\e}{\epsilon}
\newcommand{\p}{\partial}
\newcommand{\s}{\sigma}

\numberwithin{equation}{section}

\makeatletter
\def\@cite#1#2{[\textbf{#1}\if@tempswa, #2\fi]}
\makeatother

\title[Higher order Hardy--H\'enon equations]{Existence and non-existence results for the higher order Hardy--H\'enon equation revisited}

\def\cfac#1{\ifmmode\setbox7\hbox{$\accent"5E#1$}\else\setbox7\hbox{\accent"5E#1}\penalty 10000\relax\fi\raise 1\ht7\hbox{\lower1.05ex\hbox to 1\wd7{\hss\accent"13\hss}}\penalty 10000\hskip-1\wd7\penalty 10000\box7 }

\author[Q.A. Ng\^o]{Qu\cfac oc Anh Ng\^o$^{\ast}$}
\address[Q.A. Ng\^o]{
Graduate School of Mathematical Sciences, The University of Tokyo, 3-8-1 Komaba, Meguro-ku, Tokyo 153-8914, Japan
\\
ORCID iD: 0000-0002-3550-9689}
\email{\href{mailto: Q.A. Ng\^o <ngo@ms.u-tokyo.ac.jp>}{ngo@ms.u-tokyo.ac.jp}, \href{mailto: Q.A. Ng\^o <nqanh@vnu.edu.vn>}{nqanh@vnu.edu.vn}
}

\author[D. Ye]{Dong Ye}
\address[D. Ye]{Center for Partial Differential Equations, School of Mathematical Sciences and Shanghai Key Laboratory of PMMP, East China Normal
University, Shanghai 200062, China
--and--
IECL, UMR 7502, Universit\'e de Lorraine, 57000 Metz, France}
\email{\href{mailto: D. Ye <dye@math.ecnu.edu.cn>}{dye@math.ecnu.edu.cn}, \href{dong.ye@univ-lorraine.fr}{ dong.ye@univ-lorraine.fr}}

\begin{document}

\let\thefootnote\relax\footnote{$^{\ast}$ Permanent affiliation: University of Science, Vietnam National University, Hanoi.} 

\begin{abstract}
This paper is devoted to studies of non-negative, non-trivial (classical, punctured, or distributional) solutions to the higher order Hardy--H\'enon equations
\[
(-\Delta)^m u = |x|^\sigma u^p
\] 
in $\R^n$ with $p > 1$. We show that the condition
\[
n - 2m - \frac{2m+\sigma}{p-1} >0
\]
is necessary for the existence of distributional solutions. For $n \geq 2m$ and $\sigma > -2m$, we prove that any distributional solution satisfies an integral equation and a weak super polyharmonic property. We establish some sufficient conditions for punctured or classical solution to be a distributional solution. As application, we show that if $n \geq 2m$ and $\sigma > -2m$, there is no non-negative, non-trivial, classical solution to the equation if
\[
1 < p < \frac{n+2m+2\sigma}{n-2m}. 
\]
At last, we prove that for for $n > 2m$, $\s > -2m$ and $$p \geq \frac{n+2m+2\sigma}{n-2m},$$ there exist positive, radially symmetric, classical solutions to the equation. 
\end{abstract}

\date{\bf \today \; at \, \currenttime}



\maketitle

\section{Introduction}
In this note, we are interested in non-negative, non-trivial solutions to the following higher order elliptic equation
\begin{subequations}\label{eqMAIN}
\begin{align}
(-\Delta)^m u = |x|^\sigma u^p
\tag*{ \eqref{eqMAIN}$_\sigma$}.
\end{align}
\end{subequations}
in $\R^n$ with $m \geq 2$, $p>1$ and $\sigma \in \R$. Traditionally, the equation \eqref{eqMAIN}$_\sigma$ with $m =1 $ is called the H\'enon (resp. Hardy or Lane--Emden) equation if $\sigma > 0$ (resp. $\sigma <0$ or $\sigma =0$). In the same way, for $m > 1$, we call \eqref{eqMAIN}$_\sigma$ the higher order H\'enon, Hardy, or Lane--Emden equation following the sign of $\sigma$. Since we are mostly interested in $\sigma \ne 0$ and $m \geq 2$, we shall call \eqref{eqMAIN}$_\sigma$ the higher order Hardy--H\'enon equation.

In the literature, equations of the form \eqref{eqMAIN}$_\sigma$ have captured a lot of attention in the last decades, since they can be seen from various geometric and physics problems. To tackle \eqref{eqMAIN}$_\sigma$, various type of solutions were introduced and studied such as classical solutions, weak solutions, distributional solutions, singular solutions etc. For the reader's convenience, let us precise the notion of solutions that we are interested in here. First, a function $u$ is called a \textbf{classical solution} to \eqref{eqMAIN}$_\sigma$ if it belongs to the class
\begin{equation*}\label{eqClassicalSolutionClass}
C^{2m} (\R^n) \;\; \text{ if } \sigma \geq 0, \quad
C(\R^n) \cap C^{2m} (\R^n \backslash \{ 0 \} ) \;\; \text{ if } \sigma < 0;
\end{equation*}
hence the equation \eqref{eqMAIN}$_\sigma$ is verified pointwisely except eventually at $x=0$ if $\sigma < 0$. 
Second, often we can drop the definition of $u$ at the origin, namely we only require that 
\[
u \in C^{2m} (\R^n \backslash \{ 0 \} ).
\]
In this case, $u$ is called a \textbf{punctured solution} to \eqref{eqMAIN}$_\sigma$. A typical example of punctured solutions is the \textit{standard singular solution} to \eqref{eqMAIN}$_\sigma$ in the form $C_0 |x|^{-\theta}$ for suitable $p>1$, $\sigma > -2m$ and $C_0>0$. Here and after, the constant $\theta$ is as follows
\begin{align}
\label{new61}
\theta := \frac{2m + \sigma}{p-1}
\end{align}
which, as we shall see, play an important role in the work. At last, we call $u$ a \textbf{distributional solution} to \eqref{eqMAIN}$_\sigma$ if 
\begin{equation*}\label{eqDistributionalSolutionClass}
u \in L_{\rm loc}^1 (\R^n), \quad 
|x|^\sigma u^p \in L_{\rm loc}^1 (\R^n),
\end{equation*}
and \eqref{eqMAIN}$_\sigma$ is satisfied in the sense of distributions, that is,
\[
 \int_{\R^n} u (-\Delta)^m \varphi = \int_{\R^n} |x|^\sigma u^p \varphi, \quad \forall\; \varphi \in C_0^\infty(\R^n).
\]
In the literature, a distributional solution is sometimes called \textit{very weak} solution to distinguish with \textit{weak} solutions belonging to suitable Sobolev space required by variational approach. To avoid repetition, let us presume throughout this paper that
$$
\mbox{\it by a solution $u$, we always mean that $u$ is \textbf{non-negative and non-trivial}}.
$$
\indent Now let us briefly go through some literature review for classical and punctured solutions to the equation \eqref{eqMAIN}$_\sigma$. There are two important numbers, $\psnma$ and $\pcnma$ associated with \eqref{eqMAIN}$_\sigma$, called respectively the critical Sobolev and Serrin exponents, which are given by
\[
\psnma = 
\begin{cases}
\dfrac{n+2m+ 2\sigma}{n-2m} & \text{ if } n > 2m,\\
+\infty & \text{ if } n \leq 2m,
\end{cases}
\] 
and
\[
\pcnma = 
\begin{cases}
\dfrac{n + \sigma}{n-2m} & \text{ if } n > 2m,\\
+\infty & \text{ if } n \leq 2m.
\end{cases}
\] 
These critical exponents, generalizing the classical ones for the case $\sigma = 0$, are important because the solvability of the equation \eqref{eqMAIN}$_\sigma$ often changes when $p$ passes through them.

For the autonomous case, i.e.~$\sigma = 0$, the existence of classical solutions to the equation \eqref{eqMAIN}$_0$ is well-understood for general $m$. It is well-known that \eqref{eqMAIN}$_0$ has no classical solution if
\[
1 < p < \psnmz,
\]
see \cite{GS81, CLi91} for $m=1$, \cite{Lin98, Xu00} for $m=2$, and \cite{WX99} for arbitrary $m$. On the other hand, if $p \geq \psnmz$, the equation \eqref{eqMAIN}$_0$ always has positive classical solutions; see for instance \cite{GS81, Lin98, WX99, LGZ06b}. For interested readers, we also refer to \cite{NNPY18} for exhaustive existence and non-existence results of classical solutions to $\Delta^m u = \pm u^p$ in $\R^n$, with all $n, m \geq 1$ and $p \in \R$.

The situation $\sigma \ne 0$ is also well-known for the Laplacian. Fix $m = 1$, Ni proved in \cite{Ni82, Ni86} the existence of classical solution for $p\geq \ps (1, \sigma)$ with $\sigma > -2$. Hence, we are left with the subcritical case $p < \ps(1,\sigma)$. As far as we know, the subcritical case was firstly classified by Reichel and Zou. In \cite{RZ00}, they considered a cooperative semilinear elliptic system with a new development of the moving spheres method. Among others, the result of Reichel and Zou indicates that \eqref{eqMAIN}$_\sigma$ with $m = 1$ does not admit any classical solution if $1<p<\ps(1, \s)$ and $\sigma > -2$; see \cite[Theorem 2]{RZ00}. The non-existence result of Reichel and Zou was revisited by Phan and Souplet in \cite[Theorem 1.1]{PhanSouplet}; and a new proof of non-existence of \textit{bounded} solutions in the case $n=3$ was provided by using the technique introduced in \cite{SZ96}. Recently, Guo and Wan study the case of quasilinear equations in \cite{GW17}. On the other hand, as indicated by Mitidieri and Pohozaev in \cite[Theorem 6.1]{MP01}, Dancer, Du and Guo in \cite[Theorem 2.3]{DDG11} (see also Brezis and Cabr\'e \cite{BC98}), the condition $\sigma > -2$ is necessary for the existence of punctured solutions to \eqref{eqMAIN}$_\sigma$ in the case $m=1$. Thus we have a complete picture for the existence problem of classical solutions to \eqref{eqMAIN}$_\sigma$ with $m=1$ and $p > 1$.

For general polyharmonic situation $m \geq 2$, the existence of solutions to \eqref{eqMAIN}$_\sigma$ with $\sigma \ne 0$ is less understood. As above, it is natural to expect that $-2m$ serves as a threshold for $\sigma$. Remarkably, Mitidieri and Pohozaev confirmed this fact by showing that if $\sigma = -2m$, then there is no punctured super-solution to \eqref{eqMAIN}$_{-2m}$ for any $p > 1$ and any $n, m \geq 1$, even in the distributional sense; see \cite[Theorem 9.1]{MP01}. Therefore, we will always assume $\sigma \ne -2m$ throughout the paper.

There are various attempts to generalize Reichel--Zou's result to polyharmonic operator $m > 1$, see for example \cite{CL16, DaiQin-Liouville-v6} and the references there in. In the situation of \eqref{eqMAIN}$_\sigma$, it is natural to consider the system of $(u, -\Delta u, ... , (-\Delta)^{m-1}u)$. However, notice that we cannot directly apply the non-existence result of Reichel and Zou since the sign of intermediate $(-\Delta)^i u$ are unknown yet. Indeed, as remarked already in \cite{Lin98, Xu00, WX99} for $\sigma = 0$ case, the observation 
\[
(-\Delta)^i u > 0 \quad \text{ for all } \; 1 \leq i \leq m-1,
\]
which is the so-called \textit{super polyharmonic property} (SPH property for short), is essential in the study of polyharmonic equations. If the above SHP property actually holds for classical solutions to \eqref{eqMAIN}$_\sigma$ in $\R^n \backslash\{0\}$, then we can apply the non-existence result in \cite{RZ00} to get a Liouville result for classical solutions to \eqref{eqMAIN}$_\sigma$ for $\sigma > -2$. 

From now on, we shall refer the above pointwise estimate as the \textit{strong SPH property} to distinguish with a weak version to be precised later. Technically, the strong or weak SPH property serves an important tool as a kind of maximum principle which usually lacks due to polyharmonic operator. In a series of papers starting from \cite{ChenLi13}, Chen, Li and their collaborators proposed an interesting approach to study polyharmonic equations. They showed that one can transfer a differential equation to a suitable integral equation if the strong SPH property is valid. More precisely, the SPH property, if holds, is crucial in order to transform \eqref{eqMAIN}$_\sigma$ into the integral equation (for $n > 2m$)
\begin{equation*}
u(x) = C(2m)\int_{\R^n} \frac{ |y|^\sigma u^p(y)}{|x-y|^{n-2m} } dy
\end{equation*}
by using Chen--Li's trick. Here $C(\alpha)$ denotes 
\begin{equation}\label{eq-CAlpha}
C(\alpha) := \Gamma \big( \frac {n-\alpha}2 \big) \Big[ 2^\alpha \pi^{n/2} \Gamma \big(\frac \alpha 2 \big) \Big]^{-1}, \quad \forall\; 0 < \alpha < n,
\end{equation}
and $\Gamma$ is the Riemann Gamma function. Therefore, in order to study \eqref{eqMAIN}$_\sigma$, it is tempting to understand the strong SPH property of solutions, at least for the \textit{full range} $\sigma > -2m$. 

To the best of our knowledge, the first work considering the strong SPH property to classical solutions to \eqref{eqMAIN}$_\sigma$ with $\sigma \ne 0$ is due to Lei \cite[Theorem 2.1]{Lei13} in which the case $-2< \sigma \leq 0$ was examined. The case $\sigma \geq 0$ was studied by Fazly, Wei and Xu in \cite{FWX15} for $m=2$; by Cheng and Liu in \cite{CL16} for arbitrary $m \geq 1$. Recently, Dai, Peng and Qin proved in \cite{DaiPengQin-Liouville-v4} the strong SPH property for classical solutions with
\begin{align}\label{dpq18}
\left\{
\begin{aligned}
&\text{ either } \; -2-2p \leq \sigma < 0; \\
&\text{ or }\; -2m < \sigma < 0 \; \text{ and } \; u(x) = o(|x|^2) \; \text{ as }\; |x| \to \infty.
\end{aligned}
\right.
\end{align}
Clearly, when $m>1+p$ and without assuming $u(x) = o(|x|^2)$ at infinity, there is a gap
\[
-2m < \sigma < -2-2p
\]
for $\sigma$, which is not covered in \cite{DaiPengQin-Liouville-v4}. In other words, fixing any $\sigma \in ({-2m}, 0)$ and $m \geq 3$, the previous works cannot cover the range $1<p<1-\sigma/2$. This limitation is one of our motivations to work on \eqref{eqMAIN}$_\sigma$.

Unlike most of existing works in the literature, which were mainly concentrated on the classical solutions, in this work we pursue a very different route. Roughly speaking, we would like to understand more about the distributional solutions, including  connections between the three classes of solutions mentioned above. 

Following this strategy, we first show a very general non-existence result for distributional solutions to \eqref{eqMAIN}$_\sigma$.

\begin{theorem}[Liouville result for distributional solutions]\label{thm-GeneralExistence-Distributional}
Let $n, m \geq 1$, $p > 1$ and $\sigma \in \R$. If
\begin{align}
\label{new002}
n - 2m - \theta \leq 0,
\end{align}
then \eqref{eqMAIN}$_\sigma$ has no distributional solution.
\end{theorem}

Clearly, an immediate consequence of Theorem \ref{thm-GeneralExistence-Distributional} is that under the condition $\sigma > -2m$, if
\[
\left\{
\begin{aligned}
&\text{ either } \; n \leq 2m\\
&\text{ or } \; n > 2m \; \text{ and } \; 1< p \leq \pcnma,
\end{aligned}
\right.
\]
then no distributional solution exists for \eqref{eqMAIN}$_\sigma$. As an interesting application, we obtain a Liouville result for classical solutions in the critical case $n=2m$.
\begin{proposition}\label{n=2m}
Let $n = 2m \geq 2$, $\sigma > -2m$, and $p > 1$, then there is no classical solution to \eqref{eqMAIN}$_\sigma$.
\end{proposition}
Notice that the Liouville result for $n=2m$ was already obtained in \cite[Theorem 1.1]{CDQ18} under the extra condition \eqref{dpq18} when $\s$ is negative. Our observation is that when $n = 2m$ and $\sigma > -2m$, any classical solution to \eqref{eqMAIN}$_\sigma$ is a distributional one, see Lemma \ref{lem-Classic->Distribution} below. Notice that for polyharmonic equation $m \geq 2$, a classical solution is not always a distributional one, see Remark \ref{rem:j1} below.

Another interesting point we want to draw is that the condition 
$n-2m-\theta >0$ 
is sufficient and necessary for a punctured or classical solution to \eqref{eqMAIN}$_\sigma$ to be also a distributional solution; see Proposition \ref{prop-Strong->Distribution} below. 

We hope to understand more about distributional solutions to \eqref{eqMAIN}$_\sigma$. To this purpose, we use a general approach developed in the work of Caristi, D'Ambrosio and Mitidieri \cite{CAM08}, where the authors proposed an idea to gain the integral representation formula for any distributional solutions to the general equation
\begin{equation*}\label{eq-CAM}
(-\Delta)^m u = \mu
\end{equation*}
with a Radon measure $\mu$. Following the approach developed in \cite{CAM08}, we can claim 

\begin{proposition}[integral equation]\label{prop-IntegralEquation}
Let $n>2m$, $\sigma > -2m$, and $p>1$, any distributional solution $u$ to \eqref{eqMAIN}$_\sigma$ solves the integral equation
\begin{align}
\label{eqIntegralEquation}
u(x) = C(2m) \int_{\R^n} \frac{ |y|^\sigma u^p(y)}{|x-y|^{n-2m} } dy
\end{align}
for almost everywhere $x \in \R^n$. Here $C(2m)$ is the constant given by \eqref{eq-CAlpha}.
\end{proposition}

Combining with Theorem \ref{thm-GeneralExistence-Distributional} and Lemma \ref{lem-Classic->Distribution} below, the above representation formula \eqref{eqIntegralEquation} yields the following Liouville theorem for classical solutions to \eqref{eqMAIN}$_\sigma$, which significantly improves the existing results on this subject.

\begin{theorem}[Liouville result for classical solutions]\label{thm-Liouville-classical-upper}
Let $n \geq 2m$, $\sigma > -2m$, and
\[
1 < p < \psnma.
\] 
Then the equation \eqref{eqMAIN}$_\sigma$ does not admit any classical solution.
\end{theorem}

Theorem \ref{thm-Liouville-classical-upper} is \textit{optimal} seeing Theorem \ref{thm-classical-supercritical} below. In view of Theorem \ref{thm-GeneralExistence-Distributional}, it is obvious that Theorem \ref{thm-Liouville-classical-upper} is not true for distributional solutions or punctured solutions to \eqref{eqMAIN}$_\sigma$, because of the example $C_0|x|^{-\theta}$ with $\pcnma < p < \psnma$.

Another consequence of the representation formula \eqref{eqIntegralEquation} is the following result for classical solutions, which also generalizes the previous works.

\begin{proposition}[strong SPH property]\label{prop-StrongSPH}
Let $n>2m$, $\sigma > -2m$, and $p \geq \psnma$. Then any classical solution $u$ to \eqref{eqMAIN}$_\sigma$ enjoys the strong SPH property, namely
\begin{align}
\label{SPHi}
(-\Delta)^{m - i} u > 0 \quad \mbox{in }\; \R^n\backslash\{0\}
\end{align}
for any $1 \leq i \leq m-1$. Moreover, $u$ is positive in $\R^n$.
\end{proposition}

Following the argument leading to the proof of Proposition \ref{prop-StrongSPH}, the strong SPH property \eqref{SPHi} actually holds for any $p>1$. However, in view of Theorem \ref{thm-Liouville-classical-upper}, it is not necessary to treat the case $1<p<\psnma$. 

In section \ref{apd-NewProofSPH}, we will show a result on strong SPH property more general than Proposition \ref{prop-StrongSPH} above, by using completely different arguments; see Theorem \ref{SPHbis} below. A quick consequence of this more general SPH property indicates that \eqref{SPHi} still holds for some $\sigma < -2m$.

Now we turn our attention to the case $p \geq \psnma$ with necessarily $n > 2m$. In this regime, we shall establish the following existence result.

\begin{theorem}[existence for classical solutions]\label{thm-classical-supercritical}
Let $n > 2m$ and $\sigma>-2m$. For any $p \geq \psnma$, the equation \eqref{eqMAIN}$_\sigma$ always admits radial positive classical solutions.
\end{theorem}

As can be easily recognized, the existence result in Theorem \ref{thm-classical-supercritical} consists of two cases: $p=\psnma$ and $p > \psnma$. For the former situation, the result is well-known since it is related to the existence of optimal functions for higher order Hardy-Sobolev inequality; see \cite[Section 2.4]{Lions-1985}. 

For the later case, as mentioned earlier, it is a natural generalization of a classical result in \cite{Ni86} for $m=1$. To obtain such an existence result, Ni first used a fixed point argument to obtain a local solution, and followed an interesting comparison argument to realize that such a solution is indeed a global one. For $m=2$ and $\sigma = 0$, it was obtained in \cite{GG06}, where shooting method was applied with suitable $\Delta u(0)$. To ensure that such a solution is actually global, the authors used a comparison principle for polyharmonic operator given in \cite{KR}. The existence result with arbitrary $m \geq 1$ was proved in \cite{LGZ06b} for $\sigma =0$ and $p > \ps (m, 0)$. The case $\sigma > -2$ and $p>\psnma$ could be covered by Villavert's approach in \cite{Vil14}; see also \cite{LV16}. As far as we know, there is no general proof for $m \geq 2$ and $\sigma > -2m$.


Our paper is organized as follows:
\tableofcontents


\section{Preliminaries}
\label{sec-preliminaries}

Throughout the paper, by $B_R$ we mean the Euclidean ball with radius $R > 0$, centered at the origin. For brevity, the following notation for volume and surface integrals shall be used:
\[
\int_{B_R} v := \int_{B_R} v(x) dx,\quad
\int_{\partial B_R} w := \int_{\partial B_R} w(x) d\sigma_x.
\]
Since our approach is based on integral estimates, frequently we make use the following cut-off function: Let $\psi$ be a smooth radial function satisfying
\begin{align}
\label{cutoff}
{\mathbbm 1}_{B_1} \leq \psi \leq {\mathbbm 1}_{B_2}.
\end{align}
We use also the following convention: for any function $f$
\[
\Delta^{k/2} f = 
\begin{cases}
\Delta^{k/2} f & \text{ if $k$ is even},\\
\nabla \Delta^{(k-1)/2} f & \text{ if $k$ is odd}.
\end{cases}
\]
In this section, we establish several elementary estimates. We start with $L^p$-estimate for distributional and punctured solutions. Such estimates can be called Serrin--Zou type estimates; see \cite{SZ96}.
\begin{lemma}\label{lem:L1Estimate-u^p}
Let $p > 1$ and $u$ be a distributional solution to \eqref{eqMAIN}$_\sigma$, then we have
\[
\int_{B_R} |x|^\sigma u^p \lesssim R^{n-2m -\theta}, \quad \mbox{for any $R>0$.}
\]
\end{lemma}

\begin{proof}
Let $R>0$ and consider the following function 
\[
\phi_R (x) = \Big[ \psi \big( \frac x R \big) \Big]^q,
\]
where $q := 2mp/(p-1)$. Clearly
\[
|\Delta^m \phi_R (x)| 
\lesssim R^{-2m} \Big[ \psi \big( \frac xR \big)\Big]^{q-2m}
=R^{-2m} \phi_R^{1/p}.
\]
Testing the equation \eqref{eqMAIN}$_\sigma$ with the smooth function $\phi_R$, we obtain
\begin{align*}
\int_{\R^n} |x|^\sigma u^p \phi_R 
&= \int_{\R^n} u (-\Delta)^m \phi_R \\
&\leq \int_{B_{2R} \backslash B_R } u \big| (-\Delta)^m \phi_R \big|
\lesssim R^{-2m} \int_{B_{2R} \backslash B_R} u \phi_R^{1/p}.
\end{align*}
Keep in mind that in $B_{2R} \backslash B_R$ there holds
\[
|x|^{-\frac \sigma {p-1}} \leq
\begin{cases}
(2R)^{-\frac \sigma {p-1}} & \text{ if } \sigma \leq 0,\\
R^{-\frac \sigma {p-1}} & \text{ if } \sigma > 0 .
\end{cases}
\] 
Now application of H\"older's inequality gives
\begin{align*}
\int_{B_{2R} \backslash B_R} u \phi_R^{1/p}
&\leq \Big(\int_{B_{2R} \backslash B_R} |x|^{-\frac \sigma {p-1}} \Big)^{(p-1)/p}
\Big(\int_{B_{2R} \backslash B_R} |x|^\sigma u^p \phi_R \Big)^{1/p}\\
&\lesssim R^\frac{n(p-1)-\sigma}{p}
\Big(\int_{B_{2R} \backslash B_R} |x|^\sigma u^p \phi_R \Big)^{1/p}\\
&\lesssim R^\frac{n(p-1)-\sigma}{p}
\Big(\int_{B_{2R} } |x|^\sigma u^p \phi_R \Big)^{1/p} .
\end{align*}
Hence 
\[
\int_{B_{2R}} |x|^\sigma u^p \phi_R 
\lesssim R^{n-\frac{2mp+ \sigma}{p-1}}
=R^{n-2m -\theta}
\]
for any $R>0$ as claimed.
\end{proof}

With exactly the same idea but a different cut-off function $${\mathbbm 1}_{B_2\backslash B_1} \leq \xi \leq {\mathbbm 1}_{B_3\backslash B_{1/2}}$$ instead of $\psi$, we get the following estimate for punctured solutions, so we omit the proof. 

\begin{lemma}\label{lem:new1}
Let $u$ be a punctured solution to \eqref{eqMAIN}$_\sigma$ with $p > 1$, then 
\[
\int_{B_{2R}\backslash B_R} |x|^\sigma u^p \lesssim R^{n-2m -\theta}, \quad \mbox{for any $R>0$.}
\]
\end{lemma}

Next we will prove an $L^1$-estimate for any $\Delta^i u$ on $B_R$ with $1 \leq i \leq m-1$. To this aim, we make use of the following interpolation inequality on $B_R$.

\begin{lemma}\label{lem:InterpolationInequality}
Let $u$ be a non-negative function such that $u \in L^1_{\rm loc}(\R^n)$ and $\Delta^m u \in L^1_{\rm loc}(\R^n)$. Then $\Delta^i u \in L^1_{\rm loc}(\R^n)$ for any $1 \leq i \leq m-1$. Furthermore, we have
\begin{align}
\label{new001}
\int_{B_{R/2}} |\Delta^i u | 
\lesssim R^{2m-2i} \int_{B_R} |\Delta^m u | + R^{-2i}\int_{B_R \backslash B_{R/2}} u 
\end{align}
for any $R>0$.
\end{lemma}

\begin{proof}
The fact $\Delta^i u \in L^1_{\rm loc}(\R^n)$ for any $1 \leq i \leq m-1$ is standard. For any $R > 0$, consider the equation
$\Delta^m v = \Delta^m u$ in $B_R$ with the Navier boundary conditions. Then $v \in W^{2m+1, q}(B_R)$ for suitable $q>1$. Moreover, $v-u$ is a polyharmonic function, hence smooth in $B_R$; see \cite{Mit18}. Therefore, $\Delta^i u$ is locally integrable in $B_R$ for any $1 \leq i \leq m-1$. More precisely, if $R = 1$, there exists $C > 0$ such that 
\begin{equation}\label{InterpolationInequality}
\sum_{i=1}^{m-1} \int_{B_{3/4}} |\Delta^i u_k | 
< C \Big( \int_{B_1} |\Delta^m u_k | + \int_{B_1 } |u_k | \Big).
\end{equation}
\indent Now we move to \eqref{new001}. By a density argument, it suffices to establish the inequality for $u \in C^{2m} (\overline B_R)$. By a simple scaling argument, it suffices to consider the case $R=1$, namely, we wish to prove
\begin{align*}
\sum_{i=1}^{m-1} \int_{B_{1/2}} |\Delta^i u | 
\lesssim \int_{B_1} |\Delta^m u | + \int_{B_1 \backslash B_{1/2}} |u| 
\end{align*}
for any $u \in C^{2m} (\overline B_1)$. If the above claim was wrong, there would exist a sequence $(u_k) \in C^{2m}(\overline B_1)$ such that
\begin{align*}
\sum_{i=1}^{m-1} \int_{B_{1/2}} |\Delta^i u_k | 
> k \Big( \int_{B_1} |\Delta^m u_k | + \int_{B_1 \backslash B_{1/2}} |u_k| \Big), \quad \forall\; k \in {\mathbb N}.
\end{align*}
Seeing \eqref{InterpolationInequality}, there holds
\begin{align*}
 \int_{B_1} |\Delta^m u_k | + \int_{B_1 } |u_k | 
 &> \frac k C \Big( \int_{B_1} |\Delta^m u_k | + \int_{B_1 \backslash B_{1/2}} |u_k| \Big).
\end{align*}
Clearly, we can assume that $\| u_k \|_{L^1(B_{1/2})} = 1$ by scaling. Hence we get, for large $k$
\begin{equation}\label{new2}
1 = \int_{B_{1/2}} |u_k | 
\geq \frac k C \Big( \int_{B_1} |\Delta^m u_k | + \int_{B_1 \backslash B_{1/2}} |u_k| \Big).
\end{equation}
Again using \eqref{InterpolationInequality} and standard elliptic estimates, $(u_k)$ is bounded in $W^{2m, 1}(B_{5/8})$. Therefore, up to a subsequence, $u_k$ converges weakly to some $u_* \in W^{2m-1, q}(B_{5/8})$ for $1 < q < n/(n-1)$. Applying Sobolev's compact embedding and \eqref{new2}, $u_*$ enjoys
\[
\int_{B_{1/2}} |u_*| = 1, \quad \int_{B_{5/8} \backslash B_{1/2}} |u_*| = 0, \quad \Delta^m u_* = 0 \; \mbox{ in $B_{5/8}$}.
\]
In particular, $u_*$ is polyharmonic hence real analytic in $B_{5/8}$. The fact $u_* = 0$ in $B_{5/8} \backslash B_{1/2}$ yields then $u_* \equiv 0$ in $B_{5/8}$. However, this is impossible since there holds $\|u_*\|_{L^1(B_{1/2})} = 1$. So we are done.
\end{proof}

A direct consequence of the interpolation formula \eqref{InterpolationInequality} is the $L^1$-estimate for $\Delta^{i} u$ with $1 \leq i \leq m-1$.

\begin{lemma}\label{lem:L1Estimate-Delta^i}
Let $u$ be a distributional solution to \eqref{eqMAIN}$_\sigma$ in $\R^n$ with $p > 1$. Then we have $\Delta^{i} u \in L^1_{\rm loc}(\R^n)$ and
\[
\int_{B_R} |\Delta^{i} u| \lesssim R^{n-2i-\theta}
\]
for any $R>0$ and any $1 \leq i \leq m$.
\end{lemma}

\begin{proof}
Since $u$ is a distributional solution to \eqref{eqMAIN}$_\sigma$, we know that $u \in L^1_{\rm loc}(\R^n)$ and $\Delta^m u \in L^1_{\rm loc}(\R^n)$. The case $i = m$ is given by Lemma \ref{lem:L1Estimate-u^p}. Let $1 \leq i \leq m-1$, we can apply Lemma \ref{lem:InterpolationInequality} to see that $\Delta^{m-i} u \in L^1_{\rm loc}(\R^n)$ and
\begin{align*}
\int_{B_R} |\Delta^{m-i} u| 
 \lesssim R^{2i} \int_{B_{2R}} |x|^\sigma u^p + R^{2i-2m} \int_{B_{2R} \backslash B_R} u .
\end{align*}
As in the proof of Lemma \ref{lem:L1Estimate-u^p}, we can use H\"older's inequality to claim
\begin{align*}
\int_{B_{2R} \backslash B_R} u \lesssim R^\frac{n(p-1)-\sigma}{p}
\Big(\int_{B_{2R} } |x|^\sigma u^p \Big)^{1/p} .
\end{align*}
Finally, Lemma \ref{lem:L1Estimate-u^p} permits to conclude the proof.
\end{proof}

\begin{remark}\label{rmk-EllipticEstimate-NoSizeSigma}
It is important to note that we do not assume any condition on $n$, $m$ or the value of $\sigma$ in Lemmas \ref{lem:L1Estimate-u^p}-- \ref{lem:L1Estimate-Delta^i} above. 
\end{remark}


\section{Liouville result for distributional solutions}
\label{subsec-Liouville-distribution}

We prove here Theorem \ref{thm-GeneralExistence-Distributional}, namely, under the condition \eqref{new002}, the equation \eqref{eqMAIN}$_\sigma$ does not have any distributional solution. 

\begin{proof}[Proof of Theorem \ref{thm-GeneralExistence-Distributional}]
Recall the condition \eqref{new002}: $n-2m-\theta \leq 0$. If $n-2m-\theta < 0$, the desired result simply follows from Lemma \ref{lem:L1Estimate-u^p}. Therefore, we are left with the case $n-2m-\theta =0$. In this scenario, Lemma \ref{lem:L1Estimate-u^p} gives
\[
\int_{\R^n} |x|^\sigma u^p < +\infty.
\]
In particular, there holds
\[
\lim_{R\to +\infty} \int_{B_{2R} \backslash B_R} |x|^\sigma u^p =0.
\]
To derive a contradiction, we take a closer look at the proof of Lemma \ref{lem:L1Estimate-u^p}. More precisely, the following estimate
\begin{align*}
\int_{B_R} |x|^\sigma u^p 
\lesssim R^{-2m} \int_{B_{2R} \backslash B_R} u \phi_R^{1/p} \lesssim R^{-2m + \frac{n(p-1)-\sigma}{p}}
\Big(\int_{B_{2R} \backslash B_R} |x|^\sigma u^p \phi_R \Big)^{1/p}
\end{align*}
remains valid. As now
$$-2m+ \frac{n(p-1)-\sigma}{p} = \frac{p-1}{p}(n - 2m - \theta) = 0,$$
Sending $R \to +\infty$, we get $u \equiv 0$ almost everywhere. The proof of Theorem \ref{thm-GeneralExistence-Distributional} is now complete.
\end{proof}

Now we examine the condition \eqref{new002} in detail. In the case $\sigma > -2m$, \eqref{new002} is fulfilled if either $n \leq 2m$; or $n>2m$ and a $1<p \leq \pcnma$ holds. Indeed, for $n > 2m$ and $p > 1$, 
\[
p>\pcnma \quad \iff \quad n-2m-\theta >0.
\]
Hence, it remains is to understand if a distributional solution exists when $n>2m$ and $p>\pcnma$. The answer is easily positive, which means the sharpness of the threshold $\pcnma$ under the condition $\sigma > -2m$, for the existence of distributional solutions to \eqref{eqMAIN}$_\sigma$. 

Mote precisely, a simple calculation shows that in $\R^n\backslash\{0\}$, 
\begin{equation}\label{eqPolyLaplacianOnTestFunction}
\begin{aligned}
(-\Delta)^m (|x|^{-\theta}) = \prod_{k=0}^{m-1} (\theta + 2k) \times \prod_{k=1}^{m} (n-2k-\theta) |x|^\sigma |x|^{-\theta p}.
\end{aligned}
\end{equation}
Since $\theta > 0$ if $\sigma > -2m$ and $p > 1$, the first product term is positive. As $n-2m-\theta > 0$, the second product term is also positive. Thus, there exists $C_0>0$ such that $C_0|x|^{-\theta}$ is a punctured solution to \eqref{eqMAIN}$_\sigma$. Using direct verification, or Proposition \ref{prop-Strong->Distribution} below, we can check that $C_0|x|^{-\theta}$ is also distributional solution to \eqref{eqMAIN}$_\sigma$ if $\sigma > -2m$ and $p>\pcnma$. 


\section{From punctured or classical solution to distributional solution}
\label{subsec-Strong->Distribution}

We consider here the relationship between the three different types of solutions. Obviously, a classical solution is always a punctured solution. As we will soon see, for polyharmonic case $m \geq 2$, a classical solution to \eqref{eqMAIN}$_\sigma$ is not always a distributional solution. The following result provides a simple criterion to guarantee that any punctured (or classical) solution to \eqref{eqMAIN}$_\sigma$ is a distributional one. 

\begin{proposition}\label{prop-Strong->Distribution}
Suppose that $n, m \geq 1$, $p>1$, and $\sigma \in \R$. Then a punctured solution $u$ to \eqref{eqMAIN}$_\sigma$ is also a distributional solution to \eqref{eqMAIN}$_\sigma$ if and only if 
\begin{align}
\label{newj1}
n - 2m - \theta > 0.
\end{align}
The same result also holds true for classical solutions.
\end{proposition}

\begin{proof}
Seeing Theorem \ref{thm-GeneralExistence-Distributional}, we need only to prove that \eqref{newj1} is a sufficient condition. Let $u$ be a punctured solution to \eqref{eqMAIN}$_\sigma$, to prove $u \in L_{\rm loc}^1 (\R^n)$ and $|x|^\sigma u^p\in L_{\rm loc}^1 (\R^n)$, we need only to show that $u$ and $|x|^\sigma u^p$ belong to $L^1 (B_1)$. First we verify $|x|^\sigma u^p \in L^1 (B_1)$. For any $R > 0$, thanks to Lemma \ref{lem:new1}, there holds
\begin{align*}
\int_{B_{2R}\backslash B_R} |x|^\sigma u^p 
\lesssim R^{n-\frac{2mp+\sigma}{p-1}} = R^{n-2m- \theta}.
\end{align*}
If $n - 2m - \theta > 0$, using $R_k = 2^{-k}$ and summing, we get readily $|x|^\sigma u^p \in L^1 (B_1)$. 

Now, we prove $u \in L^1 (B_1)$. Note that the conditions $n -2m - \theta > 0$ and $p>1$ imply immediately $\sigma< n(p-1)$, there are two possible situations:

\noindent{\it Case 1}. If $\sigma \leq 0$, from $|x|^\sigma u^p \in L^1 (B_1)$ we immediately get $u^p \in L^1 (B_1)$, and so is $u$, as $p>1$. 

\noindent{\it Case 2}. If $0<\sigma <n(p-1)$, then by H\"older's inequality we have
\begin{align*}
\int_{B_1} u \leq \Big(\int_{B_1} |x|^{-\frac \sigma {p-1}} \Big)^{(p-1)/p}
\Big(\int_{B_1} |x|^\sigma u^p \Big)^{1/p} < +\infty,
\end{align*}
proving $u \in L^1 (B_1)$ as claimed.

Now we check that $u$ solves \eqref{eqMAIN}$_\sigma$ in the sense of distributions, equivalently
\begin{align}
\label{new4}
 \int_{\R^n} u (-\Delta)^m \varphi = \int_{\R^n} |x|^\sigma u^p \varphi 
\end{align}
holds for any $\varphi \in C_0^\infty(\R^n)$. Indeed, for each $0<\epsilon \ll 1$, consider the following cut-off function 
\[
\phi_\epsilon (x) = \Big[ 1 - \psi \big( \frac x \epsilon \big) \Big]^q,
\]
where $q = 2mp/(p-1)$ and $\psi$ is a standard cut-off function satisfying \eqref{cutoff}. Clearly, $\phi_\epsilon (x) = 0$ if $|x| \leq \epsilon$ and $\phi_\epsilon (x) =1$ if $|x| \geq 2\epsilon$. Moreover, there hold $|\nabla^k \phi_\epsilon | \leq C\epsilon^{-k}$ for all $1 \leq k \leq 2m$, thanks to $q>2m$. Using the test function $\phi_\epsilon \varphi \in C_0^\infty(\R^n\backslash\{0\})$ to \eqref{eqMAIN}$_\sigma$, we have
\begin{align*}
\int_{\R^n} \phi_\epsilon \varphi |x|^\sigma u^p
=\int_{\R^n} u (-\Delta)^m (\phi_\epsilon \varphi )
=\int_{\R^n} u \big[ \phi_\epsilon (-\Delta)^m \varphi + \Phi_\epsilon \big],
\end{align*}
where the term $\Phi_\epsilon$ enjoys
\[
|\Phi_\epsilon| \lesssim \sum_{k=1}^{2m} |\nabla^k \phi_\epsilon| | \nabla^{2m-k} \varphi| \lesssim \sum_{k=1}^{2m}\epsilon^{-k} \lesssim \epsilon^{-2m}.
\]
Note that $|\nabla^k \phi_\epsilon | \equiv 0$ outside $B_{2\epsilon} \backslash B_{\epsilon}$ for any $k \geq 1$, so is $|\Phi_\epsilon|$. Hence, we easily get
\begin{align*}
\Big| \int_{\R^n} u \Phi_\epsilon \Big| & \lesssim \epsilon^{-2m}\int_{B_{2\epsilon} \backslash B_{\epsilon}} u \\
&\lesssim \epsilon^{-2m} \epsilon^\frac{n(p-1)-\sigma}{p} \Big( \int_{B_{2\epsilon} \backslash B_{\epsilon} } |x|^\sigma u^p \Big)^{1/p}
\lesssim \epsilon^{n - 2m - \theta}.
\end{align*}
Here Lemma \ref{lem:new1} is appied for the last inequality. Therefore, 
\begin{align*}
\int_{\R^n} \phi_\epsilon \varphi |x|^\sigma u^p
=\int_{\R^n} u \phi_\epsilon (-\Delta)^m \varphi 
+ O\big(\epsilon^{n - 2m - \theta}\big).
\end{align*}
Since $|x|^\sigma u^p \in L_{\rm loc}^1 (\R^n)$, $u \in L_{\rm loc}^1 (\R^n)$, $\varphi \in C_0^\infty(\R^n)$, and $\phi_\epsilon \to 1$ a.e. as $\epsilon \to 0^+$, we can apply the dominated convergence theorem to conclude \eqref{new4}. This completes the proof for punctured solutions to \eqref{eqMAIN}$_\sigma$. 

As any classical solution to \eqref{eqMAIN}$_\sigma$ is also a punctured solution to \eqref{eqMAIN}$_\sigma$, the same conclusion is valid for classical solutions. 
\end{proof}

We should mention that when $m = 1$, no punctured solution exists if $n - 2 - \theta \leq 0$. Indeed, for $\sigma \leq -2$, \cite{MP01, DDG11} proved the non-existence for any $p > 1$; while for $\sigma > -2$, it is showed in \cite[Theorem 4.1]{GHY18} that no solution exists in any exterior domain if $1 < p \leq \pc(1,\sigma)$; see also \cite{AGQ16} for $\sigma = 0$. To conclude, we have the following fact for $m = 1$.

\begin{corollary}
Suppose that $n \geq 1$, $m=1$, $p>1$, and $\sigma \in \R$. Then any punctured solution to \eqref{eqMAIN}$_\sigma$ is also a distributional solution to \eqref{eqMAIN}$_\sigma$.
\end{corollary}

The situation is however completely different for polyharmonic equation $m \geq 2$, which shows a notable difference. Recall that when $\sigma > -2m$, the inequality $n-2m-\theta < 0$ is equivalent to $1<p < \pcnma$. In view of \eqref{eqPolyLaplacianOnTestFunction}, the function $|x|^{-\theta}$ yields a punctured solution to \eqref{eqMAIN}$_\sigma$ if and only if
\begin{align}
\label{new3}
\prod_{k=0}^{m-1} (\theta + 2k) \times \prod_{k=1}^{m} (n-2k-\theta) > 0.
\end{align}
As $\theta > 0$ in this case, and $n-2m-\theta < 0$, it suffices to select $p \in (1, \pcnma)$ such that
\[
\prod_{k=1}^{m-1} (n-2k-\theta) < 0.
\]
Apparently, this can occur for any $m \geq 2$. For example, when $m \geq 3$, a possible choice of $p > 1$ is as follows (with $n > 2m - 4$)
\[
n-2m+2-\theta<0<n-2m+4-\theta, \quad \mbox{i.e. }\; \frac{n+4+\sigma}{n-2m+4} < p < \frac{n+2+\sigma}{n-2m+2};
\]
while for $n \geq m = 2$, we can choose
\[
n-2m+2-\theta <0, \quad\mbox{i.e. }\; 1< p < \frac{n+2+\sigma}{n-2}.
\]
Another interesting remark is that even for $\sigma < -2m$ and $m \geq 2$, there exist still $p > 1$ satisfying \eqref{new3} so that $C_0 |x|^{-\theta}$ remains a punctured solution to \eqref{eqMAIN}$_\sigma$. This makes a big contrast with the non-existence result in \cite{MP01, DDG11} for $m = 1$ and $\sigma \leq -2$. For example, let $n \geq m = 2$, \eqref{new3} is satisfied by $\theta < -2$, which means that
$$\frac{4 + \sigma}{p-1} < -2, \; p > 1, \quad \mbox{i.e. }\; 1 < p < \frac{2+\sigma}{-2}, \; \sigma < -4.$$

\begin{remark}
\label{rem:j1}
Notice that with $\theta < 0$, a punctured solution $C_0 |x|^{-\theta}$ is in fact a classical solution to \eqref{eqMAIN}$_\sigma$. If we take for example $n = 5$, $m = 3$ and $\theta \in (-1, 0)$, then the condition \eqref{new3} holds true. However, the corresponding classical solution $C_0 |x|^{-\theta}$ does not satisfy \eqref{eqMAIN}$_\sigma$ in the sense of distribution, since $n - 2m - \theta <0$. 
\end{remark}

We end this section by showing another sufficient condition of different nature, which ensures also that a classical solution is a distributional one. 

\begin{lemma}
\label{lem-Classic->Distribution}
Let $u$ be a classical solution to \eqref{eqMAIN}$_\sigma$ with $p > 0$ and $\sigma > -n$. Suppose that $u$ is of class $C^k$ at the origin and $n-2m + k \geq 0$ for some $k \geq 0$, then $u$ is a also a distribution solution. In particular, if $n \geq 2m$, $\sigma > -2m$, any classical solution of \eqref{eqMAIN}$_\sigma$ is a distributional one.
\end{lemma}

\begin{proof}
We use notations similar to that in the proof of Proposition \ref{prop-Strong->Distribution}. Clearly, $u \in L_{\rm loc}^1 (\R^n)$, and there holds $|x|^\sigma u^p \in L^1_{\rm loc}(\R^n)$ since $\sigma > -n$. Hence we are left with the verification of the integral identity \eqref{new4}. Notice that we can assume $k \leq 2m-1$ as $n \geq 1$. Using integration by parts,
\begin{align*}
\int_{\R^n} \phi_\epsilon \varphi (-\Delta)^m u 
 &= (-1)^{m-k}\int_{\R^n} \Delta^{k/2} u \Delta^{m-k/2} (\phi_\epsilon \varphi)\\
& = (-1)^{m-k}\int_{\R^n} \big[\Delta^{k/2} u- (\Delta^{k/2} u) (0)\big] \Delta^{m-k/2} (\phi_\epsilon \varphi).
\end{align*}
Since $\Delta^{m-k/2} (\phi_\epsilon \varphi) = \phi_\epsilon\Delta^{m-k/2} \varphi + \Phi_\e$ for some $\Phi_\epsilon$, we obtain then
\begin{align*}
\int_{\R^n} \phi_\epsilon \varphi |x|^\sigma u^p
& = (-1)^{m-k}\int_{\R^n} \big[\Delta^{k/2} u- (\Delta^{k/2} u) (0)\big] 
\big[ \phi_\epsilon\Delta^{m-k/2} \varphi + \Phi_\e\big].
\end{align*}
Thanks to the continuity of $\Delta^{k/2} u$ at the origin and the estimate $|\Phi_\e| \lesssim \e^{k-2m}{\mathbbm 1}_{B_{2\e}\backslash B_\e}$, there holds
\begin{align*}
\left|\int_{\R^n} \big[\Delta^{k/2} u- (\Delta^{k/2} u) (0)\big] \Phi_\e\right| \leq o_\e(1) \times \e^{n-2m +k},
\end{align*}
which goes to zero as $\epsilon \to 0$, because $n -2m +k \geq 0$. Tending $\e$ to $0$, we conclude
\begin{align*}
\int_{\R^n} \varphi |x|^\sigma u^p & = (-1)^{m-k}\int_{\R^n} \big[\Delta^{k/2} u- (\Delta^{k/2} u) (0)\big] \Delta^{m-{k/2}}\varphi\\
& = (-1)^{m-k}\int_{\R^n} \Delta^{k/2} u \Delta^{m-{k/2}}\varphi\\
& = \int_{\R^n} u (-\Delta)^m\varphi.
\end{align*}
So we are done. 
\end{proof}


In practice, Lemma \ref{lem-Classic->Distribution} is quite useful since it helps us to obtain Liouville result for classical solutions via that for distributional solutions established in Theorem \ref{thm-GeneralExistence-Distributional}. For example, combining Lemma \ref{lem-Classic->Distribution} with Theorem \ref{thm-GeneralExistence-Distributional}, we easily get Proposition \ref{n=2m}. Indeed, as $n - 2m - \theta = -\theta < 0$ if $n = 2m$, $\sigma > -2m$ and $p > 1$, then no distributional solution to \eqref{eqMAIN}$_\sigma$ can exist.

In view of Theorem \ref{thm-Liouville-classical-upper}, it is natural to ask whether or not a Liouville result for classical solutions exists if $n<2m$. As far as we know, there is no such a result for $m \geq 2$. However, by using Lemma \ref{lem-Classic->Distribution}, we can conditionally obtain such a result. 

\begin{corollary}\label{cor-Liouville-n<2m}
Let $2 \leq n < 2m$, $\sigma > -2m$, and $p > 1$. Then there is no classical solution to \eqref{eqMAIN}$_\sigma$, which is of class $C^{2m-n}$ at the origin, in particular, the equation \eqref{eqMAIN}$_\sigma$ has no solution in $C^{2m-2} (\R^n) \cap C^{2m} (\R^n \backslash \{ 0 \} )$.
\end{corollary}



\section{Integral equation and the weak SPH for distributional solutions}
\label{Properties-ds}

In this section, we establish two important properties of distributional solutions. First, for $n > 2m$, $\sigma > -2m$, and $p > 1$, we will show that any distributional solution to the differential equation \eqref{eqMAIN}$_\sigma$ solves the integral equation \eqref{eqIntegralEquation} almost everywhere. Next, we show that any distributional solution to \eqref{eqMAIN}$_\sigma$ satisfies the weak SPH property; see Proposition \ref{prop-WeakSPH} below. As application, we get the strong SPH property for classical solutions to \eqref{eqMAIN}$_\sigma$, namely Proposition \ref{prop-StrongSPH}.

The departure point for us is the following result.
\begin{lemma}[ring condition]
\label{lem:ring}
Any distributional solution $u$ to \eqref{eqMAIN}$_\sigma$ with $\sigma > -2m$ and $p>1$ satisfies the ring condition:
\begin{equation}\label{eqRingCondition}
\lim_{R \to + \infty} \frac 1{R^n} \int_{R \leq |x-y| \leq 2R} u(y) dy = 0, \quad \forall\; x \in \R^n.
\end{equation}
\end{lemma}

\begin{proof}
Fix any $x \in \R^n$, consider $R > 2|x|$. Readily $\{y : R \leq |x-y| \leq 2R \} \subset B_{3R}\backslash B_{R/2}$. By H\"older's inequality and Lemma \ref{lem:L1Estimate-u^p}, there holds
\begin{align*}
\int_{R \leq |x-y| \leq 2R} u(y) & \leq \int_{B_{3R}\backslash B_{R/2}} u\\
&\leq \Big(\int_{B_{3R}\backslash B_{R/2}} |y|^{-\frac \sigma {p-1}} dy \Big)^{(p-1)/p}
\Big(\int_{B_{3R}\backslash B_{R/2}} |y|^\sigma u^p (y) dy \Big)^{1/p}\\
& \lesssim R^{(n-\frac\sigma{p-1})\frac{p-1}p} R^\frac{n-2m-\theta}p\\
& = R^{n-\theta}.
\end{align*}
Hence the distributional solution $u$ enjoys the ring condition, thanks to $\theta > 0$.
\end{proof}

\begin{remark}
\label{rem:j2}
Applying Lemma \ref{lem:new1}, the same proof shows that if $\sigma > -2m$ and $p > 1$, any punctured solution of \eqref{eqMAIN}$_\sigma$ also satisfies the ring condition \eqref{eqRingCondition}.
\end{remark}

From the ring condition \eqref{eqRingCondition} we can apply a general result of Caristi, D'Ambrosio, and Mitidieri to conclude that, any distributional solution $u$ to \eqref{eqMAIN}$_\sigma$ solves \eqref{eqIntegralEquation} almost everywhere; see \cite[Theorem 2.4]{CAM08}. However, only the proof for $m=2$ was provided in \cite{CAM08}, and we are not convinced that \eqref{eqIntegralEquation} holds for all Lebesgue points of $u$, as claimed in \cite{CAM08}. We show here a detailed proof for all $m$, for the sake of completeness and the reader's convenience.

Let us first introduce some notations. Denote
\[
{\mathbf G}^\epsilon (x) = \Big( \frac 1{\epsilon^2 + |x|^2} \Big)^\frac{n-2m}2 \quad \mbox{and} \quad U_q (x) = \Big( \frac 1{1+|x|^2} \Big)^q \quad\text{with} \; \epsilon, q > 0.
\]
We shall use the test function
\begin{align}
\label{new5}
\varphi(x) = \phi_R(x) {\mathbf G}^\epsilon(x) = \psi \big(\frac{x}{R} \big) {\mathbf G}^\epsilon(x), \quad R, \epsilon > 0;
\end{align}
where $\psi$ is a cut-off function satisfying \eqref{cutoff}. We need also some estimate for $(-\Delta)^m {\mathbf G}^\epsilon$. Toward a precise computation of this term, we use the following auxiliary result.

\begin{lemma}\label{lem-IdentityU}
There holds
\begin{align*}
(-\Delta)^m U_q =& \; 2^m\prod_{k=0}^{m-1}(q+k)\prod_{k=1}^m(n-2k-2q)U_{q+m}\\
&\; + \sum_{i=1}^{m-1}2^{m+i} \binom{m}{i} \prod_{k=0}^{m+i-1}(q+k)\prod_{k=i+1}^m(n-2k-2q)U_{q+m+i}\\
 & \; + 2^{2m}\prod_{k=0}^{2m-1}(q+k)U_{q+2m}.
\end{align*}
\end{lemma}

\begin{proof}
A direct calculation gives
\begin{align*}
-\Delta U_q(x) &= 2q\left[\frac{n(1+|x|^2) - 2(q+1)|x|^2}{(1+|x|^2)^{q+2}}\right]\\
& = 2q(n-2-2q)U_{q+1}(x) + 4q(q+1) U_{q+2}(x).
\end{align*}
Hence, the proof follows by induction on $q$. We omit its details.
\end{proof}

Having all the notations above, we can proceed now the proof of Proposition \ref{prop-IntegralEquation}.
\begin{proof}[Proof of Proposition \ref{prop-IntegralEquation}]
To prove \eqref{eqIntegralEquation} for a fixed point $x$, it suffices to verify that
\begin{equation}\label{eq-IntegralIdentity-changed}
u(x) = C(2m) \int_{\R^n} \frac{ |x-y|^\sigma u^p(x-y)}{|y|^{n-2m} } dy.
\end{equation}
Here the constant $C(2m)$ is given by \eqref{eq-CAlpha}. Testing our equation 
$$(-\Delta)^m u(x-y) = |x-y|^\sigma u^p(x-y)$$ 
with $\phi_R {\mathbf G}^\epsilon$ given by \eqref{new5}, integration by parts yields
\begin{align*}
\int_{\R^n} |x-y|^\sigma u^p(x-y) \phi_R (y) {\mathbf G}^\epsilon (y) dy 
& = \int_{\R^n} u(x-y) (-\Delta)^m \big( \phi_R {\mathbf G}^\epsilon \big) (y) dy\\ & =: I_1^\epsilon + I_2^\epsilon,
\end{align*}
where
\[
I_1^\epsilon= \int_{\R^n} u(x-y) \phi_R (y) (-\Delta)^m {\mathbf G}^\epsilon (y) dy.
\]
Using the notation $U_q$ as in Lemma \ref{lem-IdentityU}, we see that ${\mathbf G}^\epsilon (y) = \epsilon^{2m-n} U_{\frac{n-2m}2}(y/\epsilon)$ and
\begin{align*}
 (-\Delta)^m {\mathbf G}^\epsilon (y) &= \epsilon^{-n} 2^{2m}\prod_{k=0}^{2m-1} \Big( \frac{n-2m}{2}+k \Big) U_\frac{n+2m}{2} \big( \frac y \epsilon \big) \\
 &= \epsilon^{-n} 2^{2m} \frac{ \Gamma \big( \frac{n+2m}{2}\big) }{\Gamma \big( \frac{n-2m}{2}\big) } U_\frac{n+2m}{2} \big( \frac y \epsilon \big) .
\end{align*}
Clearly,
\begin{align*}
\int_{\R^n} \epsilon^{-n} U_\frac{n+2m}{2} \big( \frac y \epsilon \big) dy
& =\int_{\R^n} U_\frac{n+2m}{2} (y) dy \\
& = |\mathbb S^{n-1}|\int_0^{+\infty} \Big( \frac 1{1+r^2} \Big)^\frac{n+2m}{2} r^{n-1}dr\\
&= \frac{2\pi^{n/2}}{\Gamma \big( \frac n2 \big)} \frac{\Gamma \big( \frac{n+2m}{2} - \frac{n}{2} \big)\Gamma \big( \frac{n}{2} \big)}{2\Gamma \big( \frac{n+2m}{2} \big)}
= \frac{ \pi^{n/2} \Gamma (m) }{\Gamma \big( \frac{n+2m}{2} \big)}.
\end{align*}
So we can rewrite $I_1^\epsilon$ as follows
\begin{align*}
I_1^\epsilon &= \int_{\R^n} u (x-y) \phi_R (y) (-\Delta)^m {\mathbf G}^\epsilon (y) dy\\
& = 2^{2m} \frac{ \Gamma \big( \frac{n+2m}{2}\big) }{\Gamma \big( \frac{n-2m}{2}\big) }
\int_{\R^n} u (x - \epsilon y) \phi_R (\epsilon y) U_\frac{n+2m}{2} (y) dy \\
& = 2^{2m} \frac{ \Gamma \big( \frac{n+2m}{2}\big) }{\Gamma \big( \frac{n-2m}{2}\big) }
\big( f * g_\epsilon \big) (x) ,
\end{align*}
with $f(z) = u(z)\phi_R(x-z) \in L^1(\R^n)$ and $g_\epsilon (z) = \epsilon^{-n} U_\frac{n+2m}{2} (z/\epsilon)$. By definition, it is clear that \textit{the least decreasing radial majorant} of $U_\frac{n+2m}{2}$ is integrable, i.e.
\[
\int_{\R^n} \big[\sup_{|x| \geq |y|} U_\frac{n+2m}{2} (x) \big] dy =\int_{\R^n} U_\frac{n+2m}{2} (y) dy <+\infty.
\]
Therefore, we can apply \cite[Theorem 2(b)]{Stein} to claim that
\begin{equation}\label{eq-EstimateI_1}
\lim_{\epsilon \to 0^+} I_1^\epsilon =2^{2m} \frac{ \Gamma \big( \frac{n+2m}{2}\big) }{\Gamma \big( \frac{n-2m}{2}\big) } u(x) \phi_R (0) \int_{\R^n} U_\frac{n+2m}{2} (y)dy
= \frac 1{C(2m)} u(x)
\end{equation}
for almost everywhere $x$. On the other hand, as $u \in L^1_{loc}(\R^n)$, letting $\epsilon \to 0^+$ gives
\begin{equation*}
\lim_{\epsilon \to 0^+} I_2^\epsilon = \int_{B_{2R} \backslash B_R} u(x-y) L( \phi_R) (y) dy ,
\end{equation*}
where $L$ is the operator defined by
\[
 L : \phi \mapsto (-\Delta)^m(\phi{\mathbf G}^0) - \phi(-\Delta)^m{\mathbf G}^0= (-\Delta)^m(\phi{\mathbf G}^0), \quad \mbox{in } \R^n\backslash\{0\},
\]
with ${\mathbf G}^0(x) = |x|^{2m-n}$. Observe that $$\phi_R{\mathbf G}^0(x) = R^{2m-n}(\psi{\mathbf G}^0)(x/R),$$ we easily get 
$L( \phi_R) = R^{-n}L(\psi)(x/R)$, hence $$|L(\phi_R) | \leq CR^{-n}{\mathbbm 1}_{B_{2R} \backslash B_R},$$ where $C$ is a constant independent of $R > 0$. Consequently,
\begin{align}
\label{eq-EstimateI_2}
\big|\lim_{\e\to 0^+} I_2^\epsilon \big| \leq CR^{-n}\int_{B_{2R} \backslash B_R} u(x-y)dy.
\end{align}
Finally, for a.e. $x$ and for any $R > |x|$, tending $\e$ to $0^+$, using \eqref{eq-EstimateI_1}--\eqref{eq-EstimateI_2} and the proof of Lemma \ref{lem:ring}, we conclude that
\[
\int_{\R^n} {\mathbf G}^0(y) |x-y|^\sigma u^p(x-y) \phi_R(y) dy = \frac 1{C(2m)} u(x)+ O\big(R^{-\theta}\big)_{R \nearrow +\infty}.
\]
Sending now $R \to +\infty$, we just proved
\[
u(x) =C(2m) \int_{\R^n} {\mathbf G}^0(y) |x-y|^\sigma u^p(x-y) dy
\]
for a.e.~$x$. This completes the proof of \eqref{eq-IntegralIdentity-changed}, equivalently \eqref{eqIntegralEquation} holds for a.e. $x$.
\end{proof}

From the integral equation, it is easy to obtain the weak or strong SPH properties for solutions to (\ref{eqMAIN})$_\sigma$. These properties play no crucial role here for the existence or non-existence results, which illustrates a key difference between our approach and other approaches in the existing literature. It is also worth noting that the weak SPH property is stated for distributional solutions to the integral equation \eqref{eqIntegralEquation}, which is also of fundamental difference to the strong SPH property. The result below is not really new, it is indeed part of \cite[Theorem 2.4]{CAM08}. 

\begin{proposition}[weak SPH property]\label{prop-WeakSPH}
Let $n>2m$. Then any distributional solution $u$ to the integral equation \eqref{eqIntegralEquation} enjoys the weak SPH property, namely there hold
\[
\int_{\R^n} u (-\Delta)^{m - i} \phi \geq 0
\]
for all $1 \leq i \leq m-1$ and for any $0 \leq \phi \in C_0^\infty(\R^n)$. 
\end{proposition}

\begin{proof} 
Let $0 \leq \phi \in C_0^\infty(\R^n)$ and $1 \leq i \leq m-1$ be arbitrary. First we recall the well-known Selberg formula
\[
\int_{\R^n} \frac{C(\alpha)}{|x-z|^{n-\alpha}}\frac{C(\beta)}{|y-z|^{n-\beta}} dz
=\frac{C(\alpha + \beta)}{|x-y|^{n-\alpha-\beta}}.
\]
where $\alpha, \beta > 0$, $\alpha + \beta < n$; and $C(\alpha)$, $C(\beta)$ and $C(\alpha+\beta)$ are constants in \eqref{eq-CAlpha}; see \cite{GM99}. Using the above formula and Fubini's theorem, we can rewrite $u$ from \eqref{eqIntegralEquation} as follows
\begin{align*}
u(x)& = C(2m) \int_{\R^n} \frac{ |y|^\sigma u^p(y)}{|x-y|^{n-2m} } dy\\
&=C(2m-2i) \int_{\R^n} \frac{ 1}{|x-z|^{n-2m+2i} } \Big( C(2i) \int_{\R^n} \frac{|y|^\sigma u^p(y) }{|y-z|^{n-2i} } dy \Big) dz\\
& =C(2m-2i) \int_{\R^n} \frac{1}{|x-z|^{n-2m+2i} } d\mu_i (z)
\end{align*}
for some positive measure $\mu_i$. Now, we multiply both sides of \eqref{eqIntegralEquation} by $(-\Delta)^{m-i} \phi$ and integrate to get
\begin{align*}
\int_{\R^n} u (-\Delta)^{m - i} \phi 
& = C(2m-2i) \int_{\R^n} \Big( \int_{\R^n} \frac{1}{|x-z|^{n-2m+2i} } d\mu_i (z) \Big) (-\Delta)^{m - i} \phi \\
& = C(2m-2i) \int_{\R^n} \phi (-\Delta)^{m - i} \Big( \int_{\R^n} \frac{1}{|x-z|^{n-2m+2i} } d\mu_i (z) \Big) \\
& = \int_{\R^n} \phi (x) d\mu_i (x) \geq 0.
\end{align*}
This implies that $u$ satisfies the weak SPH property. 
\end{proof}


\section{The strong SPH property for classical solutions}
\label{apd-NewProofSPH}

In the existing literature, the method of proving the strong SPH property \eqref{SPHi} is often based on careful analysis on the spherical averages of $(-\Delta)^i u$, which could be rather technical and involved, see for example \cite{WX99, ChenLi13}.

Here we can easily obtain the strong SPH property based on its weak form. It is quite obvious to see that Proposition \ref{prop-WeakSPH} implies Proposition \ref{prop-StrongSPH}. Indeed, for $n > 2m$, $\sigma > -2m$, and $p > 1$, any classical solution is a distributional one using Lemma \ref{lem-Classic->Distribution}, hence $(-\Delta)^i u \geq 0$ in $\R^n\backslash\{0\}$ for $1\leq i \leq m-1$ using Proposition \ref{prop-WeakSPH}. Furthermore, as $(-\Delta)^m u \geq 0$ and not identically zero, the strong maximum principle ensures that all $(-\Delta)^i u$ are positive in $\R^n\backslash\{0\}$. Using just \eqref{eqIntegralEquation}, we see that $u$ is positive in $\R^n$.

In fact, we can prove a result more general than Proposition \ref{prop-StrongSPH}, by integral estimates for a distributional solution, without using the weak SPH property, nor the usual spherical average procedure. 

\begin{theorem}[{\it partially} SPH property]
\label{SPHbis}
Let $u$ be both classical and distributional solution to \eqref{eqMAIN}$_\sigma$ with $p > 1$ and $n \geq 3$. Assume that there exists $\ell \in {\mathbb N}$ such that $m \geq \ell+1$ and $2\ell + \theta > 0$, then
\[
(-\Delta)^i u > 0 \quad \text{ in } \;  \R^n\backslash\{0\} \quad  \text{ for all } \; \ell\leq i \leq m-1.
\]
In particular, the strong SPH property \eqref{SPHi} holds under $n \geq 3$, $m \geq 2$ and $\theta > -2$ where we select $\ell =1$.
\end{theorem}

Our proof is inspired by an idea from \cite[Appendix]{FWX15}, where Fazly, Wei and Xu suggested a simple argument to handle classical solutions to $\Delta^2 u = |x|^\sigma u^p$ with $\sigma \geq 0$. To show $-\Delta u \geq 0$, their idea is to estimate the harmonic function $h := \Delta u + w$ from the above, where
\[
w(x) =C(2) \int_{\R^n} \frac{|y|^\sigma u^p(y)}{|x-y|^{n-2}}dy \geq 0,
\]
and $C(2)$ is given by \eqref{eq-CAlpha}. For any $x_0 \in \R^n$, we have
\[
h(x_0) = \fint_{\partial B_R (x_0)} (\Delta u + w)
\leq \fint_{\partial B_R (x_0)}|\Delta u| + \fint_{\partial B_R (x_0)} w, \quad \forall\; R > 0.
\]
Therefore, if the right hand sides goes to zero for a suitable sequence $R_k \to +\infty$, then $h(x_0) \leq 0$, hence $-\Delta u(x_0) \geq 0$ as expected. Unfortunately, they met some difficulty in \cite{FWX15} to control $\|\Delta u\|_{L^1(\p B_R(x_0))}$ from the above.

Here we generalize the idea in \cite{FWX15}; in particular giving a new, independent proof of Proposition \ref{prop-StrongSPH}. Our proof make uses of the integral estimates. Before proving the result, recall the following fact from \cite[Section 9.7]{LiebLoss}: For any $n \geq 3$, $x_0, y \in \R^n$ and $R > 0$, there holds
\begin{equation}\label{eq-LL}
\fint_{\partial B_r (x_0)} \frac{d\sigma_y }{ |x - y|^{n-2} }
= \max \big\{ |x-x_0|, r \big\}^{2-n}.
\end{equation}

\begin{proof}[Proof of Theorem \ref{SPHbis}]
Let $u$ be a classical and distributional solution of \eqref{eqMAIN}$_\sigma$ with $p > 1$, $n \geq 3$, $m \geq \ell +1$, $\ell \in {\mathbb N}$ and $2\ell + \theta > 0$.

Fix $\ell \leq i \leq m-1$.  From Lemma \ref{lem:L1Estimate-Delta^i}, there holds
\begin{align*}
\int_{B_{2R} \backslash B_R}  |x|^{-n+2} |\Delta^{i+1}u| \lesssim R^{- 2i - \theta}, \quad \forall \; R > 0.
\end{align*}
Remark that $-2i-\theta \leq -2\ell - \theta < 0$ as $i \geq \ell$. Summing up with $R_k = 2^kR$, we get
\begin{align}
\label{newa2}
\int_{\R^n \backslash B_R} |x|^{-n+2}|\Delta^{i+1}u| \lesssim R^{-2i - \theta}, \quad \forall \; R > 0.
\end{align}
We claim that
\[
w_i (x) = C(2) \int_{\R^n}\frac{(-\Delta)^{i+1}u (y)}{|x - y|^{n-2}} dy
\]
is well defined for any $x\in \R^n\backslash\{0\}$. Indeed, let $x\ne 0$ be arbitrary but fixed point, there hold:
\begin{itemize}
\item on $B_{|x|/2}$, the integral is bounded as $(-\Delta)^{i+1}u \in L^1_{\rm loc}(\R^n)$;
\item on $B_{2|x|}\backslash B_{|x|/2}$, the integral exists since $(-\Delta)^{i+1}u$ is bounded over this set;
\item on $\R^n\backslash B_{2|x|}$, the integral is easily bounded, thanks to \eqref{newa2} and the inequality $|x-y| \geq |y|/2$.
\end{itemize}
The well-definition of $w_i$ for a.e.~$x \in \R^n$ and $(-\Delta)^{i+1}u \in L_{\rm loc}^1 (\R^n)$ allow us to apply \cite[Theorem 6.21]{LiebLoss} to deduce that $w_i$ satisfies 
\[
-\Delta w_i = (-\Delta)^{i+1}u \quad \mbox{in }\; {\mathcal D}'(\R^n).
\]
Therefore 
\[
h_i := w_i - (-\Delta)^{i+1}u
\]
solves $-\Delta h_i = 0$ in ${\mathcal D}'(\R^n)$. Hence $h_i$ is harmonic and smooth in $\R^n$ by the classical Weyl lemma, see \cite[Theorem 7.10]{Mit18} or \cite[page 256]{LiebLoss}, consequently $w_i \in C(\R^n \backslash \{ 0 \})$. Hence we have, for any $x_0 \in \R^n\backslash\{0\}$ and any $R > |x_0|$, 
\begin{align}
\label{newa1}
h_i (x_0) = \fint_{\partial B_R (x_0)} h_i 
\leq \fint_{\partial B_R (x_0)} w_i
+ \fint_{\partial B_R (x_0)}|\Delta^{i+1} u|.
\end{align}
Following the idea in \cite{FWX15}, it is necessary to estimate the two integrals on the right hand side of \eqref{newa1}. By Fubini's theorem, we can write
\begin{align*}
\frac{1}{C(2)} & \fint_{\partial B_R (x_0)} w_i d\sigma_x  \\
& \leq \int_{\R^n} \Big(\fint_{\partial B_R (x_0)} \frac{d\sigma_x }{ |x - y|^{n-2} }\Big)  |\Delta^{i+1}u(y)| dy\\
&= \Big( \int_{|y-x_0| > R} + \int_{|y-x_0| < R} \Big) \Big(\fint_{\partial B_R (x_0)} \frac{d\sigma_x }{ |x - y|^{n-2} }\Big) |\Delta^{i+1}u(y)| dy\\
& =: I_1 + I_2.
\end{align*}
For any $R > 2|x_0|$, by \eqref{eq-LL} and \eqref{newa2}, there holds
\[
I_ 1 = \int_{|y-x_0| > R} \frac{|\Delta^{i+1}u(y)|}{ |x_0 - y|^{n-2} } dy \lesssim (R-|x_0|)^{-2i - \theta}.
\]
For $I_2$, still by \eqref{eq-LL} and using Lemma \ref{lem:L1Estimate-u^p}, we deduce that
\begin{align*}
I_2 = R^{2-n} \int_{|y-x_0| < R} |\Delta^{i+1}u(y)| dy & \leq R^{2-n} \int_{B_{R+|x_0|}} |\Delta^{i+1}u(y)| dy\\
& \lesssim R^{2-n} (R+ |x_0|)^{n- 2(i+1) - \theta} \\
& \lesssim (R+ |x_0|)^{-2i - \theta}.
\end{align*}
Putting the estimates for $I_1$ and $I_2$ together, we have
\[
\lim_{R \to +\infty} \fint_{\partial B_R (x_0)} w_i = 0,
\]
thanks again to $-2i - \theta < 0$. Moreover, in view of Lemma \ref{lem:L1Estimate-Delta^i}, there exists a sequence $R_k \to +\infty$ such that 
\[
\liminf_{k \to +\infty} \fint_{\partial B_{R_k} (x_0)} |\Delta^i u| = 0.
\]
Using this sequence $R_k$ in \eqref{newa1}, we see that $h_i(x_0) \leq 0$, so $(-\Delta)^i u(x_0) > 0$ as $w_i (x_0) > 0$. 
The proof is completed.
\end{proof}

The following are some comments on Theorem \ref{SPHbis}.
\begin{itemize}
\item The condition $2\ell + \theta > 0$ is {\it almost} necessary for Theorem \ref{SPHbis}. For example, let $m = 2$, $n \geq 4$, $\theta < -2$ and $\ell = 1$, $u = C|x|^{-\theta}$ is a classical and distributional solution to \eqref{eqMAIN}$_\sigma$ with suitable $C > 0$,  but $\Delta u > 0$ in $\R^n\backslash \{0\}$.
\item By Theorem \ref{thm-GeneralExistence-Distributional}, an implicit condition in Theorem \ref{SPHbis} is $n - 2m - \theta > 0$.  
  \item Clearly, seeing Lemma \ref{lem-Classic->Distribution}, Proposition \ref{prop-StrongSPH} is a special case of Theorem \ref{SPHbis} with $\ell = 0$ and $\theta > 0$.
\end{itemize}

We call the property obtained in Theorem \ref{SPHbis} the {\it partially} SPH property, since we could have the positivity of $(-\Delta)^i u$ only for $\ell \leq i \leq m-1$ instead of the whole range $1 \leq i \leq m-1$. Theorem \ref{SPHbis} could be useful to understand classical solutions for $n \geq 2m$, $\s < -2m$; or for $n = 2m-1$. 

As far as we know, the {\it partially} SPH property for \eqref{eqMAIN}$_\sigma$ with $p>1$ has not been studied before. However, for $p<0$, it was observed in \cite{DN17} that for any positive smooth solution $u$ to $(-\Delta)^3 u = u^p$ in $\R^3$ with $-2 < p<-1/2$, where $\Delta^2 u > 0$ in $\R^3$, and it is likely that $\Delta u$ does not have a fixed sign. 

\section{Existence and non-existence of classical solutions}

This section is devoted to the proof of Theorems \ref{thm-Liouville-classical-upper} and \ref{thm-classical-supercritical}. First we start with a quick proof for Theorem \ref{thm-Liouville-classical-upper}.

\begin{proof}[Proof of Theorem \ref{thm-Liouville-classical-upper}]
The case $n=2m> -\sigma$ and $p>1$ is just Proposition \ref{n=2m}. Suppose now $n > 2m > -\s$ and $1 < p < \psnma$. Using the fact that any classical solution to \eqref{eqMAIN}$_\sigma$ is a distributional solution, it follows from the proof of Proposition \ref{prop-IntegralEquation} that $u$ satisfies the integral equation \eqref{eqIntegralEquation} everywhere and $u > 0$ in $\R^n$. Now it is standard to realize that $u$ is radially symmetric with respect to the origin; for example, see \cite[Theorem 3]{CL16} or \cite[Theorem 1.6]{DaiQin-Liouville-v6}. A consequence of the symmetry of $u$ is that it satisfies the following upper bound
\[
u(x) \lesssim |x|^{-\theta}.
\]
From this estimate, there holds
\[
\int_{\R^n} |x|^\sigma u^{p+1} < +\infty,
\]
since $p<\psnma$. From this fact and the positivity of $u$, we can easily obtain a contradiction by making use a Pohozaev type identity. 
\end{proof}


In the following, we consider the existence of classical solution for $n > 2m > -\s$ and $p \geq \psnma$. As mentioned in Introduction, the existence result in the case $p= \psnma$ is well-known because it is related to the existence of optimal functions for the following higher order Hardy--Sobolev inequality
\[
\Big( \int_{\R^n} |x|^\sigma |u|^\frac{2(n+\sigma)}{n-2m} \Big)^\frac{n-2m}{n+\sigma}
\leq C_{\mathsf{HS}} \int_{\R^n} |\Delta^{m/2} u|^2, \quad \forall\; u \in \mathcal D^{m,2}(\R^n).
\]
Recall that the space $\mathcal D^{m,2}(\R^n)$ is the completion of $C_0^\infty (\R^n)$ under the Dirichlet norm, see \cite[Section 2.4]{Lions-1985}. Recall also that $\Delta^{m/2} = \nabla\Delta^{(m-1)/2}$ when $m$ is odd. Since optimal functions for the above inequality can be characterized by 
\begin{align}
\label{new6}
\inf_{u \in \mathcal D^{m,2}(\R^n) \backslash \{ 0 \}} 
\frac
{\|\Delta^{m/2} u\|_{L^2(\R^n)}^2}
{\Big\||x|^\sigma |u|^\frac{2(n+\sigma)}{n-2m}\Big\|_{L^1(\R^n)}^\frac{n-2m}{n+\sigma}},
\end{align}
it is easy to verify that any optimal function for the Hardy--Sobolev inequality yields a distributional solution to \eqref{eqMAIN}$_\sigma$. The fact that any optimal function $u$ to \eqref{new6} belongs to $C^{2m}(\R^n \backslash \{ 0 \} ) \cap C(\R^n)$ is also well-known; for example, see \cite[Theorem 3]{JL14}. Hence, in the rest of this section, we will handle the supercritical case $p > \psnma$.

To look for a solution to \eqref{eqMAIN}$_\sigma$, it is common to establish a local existence first, which often relies on either a fixed-point argument, see \cite{Ni86, LGZ06b, Vil14}, or the shooting method, see \cite{GG06}. However it 
is not so clear how to employ a uniform fixed-point argument for \eqref{eqMAIN}$_\sigma$ in the full range of $\sigma > -2m$. It seems also difficult to apply the shooting method, since some $(-\Delta)^i u$ could have no sense at the origin for $\sigma < 0$. Consequently, to obtain the existence result for \eqref{eqMAIN}$_\sigma$, we shall use an indirect argument. 

We start with positive solutions to the following auxiliary problem
\begin{subequations}\label{eqA}
\begin{align}
\left\{
\begin{aligned}
(-\Delta)^m u &= \lambda |x|^\sigma (1+u)^p & & \text{ in } B_1,\\
\nabla^i u \big|_{\partial B_1} &=0 & & \text{ for } 0 \leq i \leq m-1,
\end{aligned}
\right.
\tag*{\eqref{eqA}$_\lambda$}.
\end{align}
\end{subequations}
with $\lambda > 0$. Under the Dirichlet boundary condition, it is well-known that the kernel for $(-\Delta)^m$ is positive on the balls, so it is not difficult to use standard method to get existence result for small $\lambda > 0$. A key observation is that for supercritical exponent $p$, the equation \eqref{eqA}$_\lambda$ admits a unique solution if $\lambda > 0$ is small enough. Then we study the set of radial solutions to \eqref{eqA}$_\lambda$ with different $\lambda$, and show that the suitable scaling of a sequence of solutions to \eqref{eqA}$_\lambda$ converges to a classical solution of \eqref{eqMAIN}$_\sigma$. This approach was recently used in \cite[Section 2]{ACDFGW} and in \cite{HS19}, respectively for fractional Laplacian and biharmonic case. 

For clarity, we formulate the proof of Theorem \ref{thm-classical-supercritical} into several subsections.

\subsection{Existence of solutions to (\ref{eqA})$_\lambda$ for all $0<\lambda \leq \lambda^*$}

Here we prove the existence of the minimal solution $u_\lambda$ to \eqref{eqA}$_\lambda$ for $0<\lambda \leq \lambda^*$, where $\lambda^*$ is a critical value to be precised later. The crucial point is that $\G$, the Green function of $(-\Delta)^m$ on $B_1$ under the Dirichlet boundary conditions, is positive. Indeed, by Boggio's formula, 
\[
\G(x,y) = k_{n,m} |x-y|^{2m-n} \int_1^\frac{\sqrt{|x|^2 |y|^2 - 2 x\cdot y +1}}{|x-y|} (t^2-1)^{m-1} t^{1-n} dt\quad \forall\; x, y \in B_1,
\]
for some constant $k_{n,m}>0$; see \cite{GGS10}. From the positivity of $\G$ we can apply the standard monotone iteration method. First, $w_0 \equiv 0$ is obviously a subsolution to \eqref{eqA}$_\lambda$ for any $\lambda > 0$. Consider 
\begin{align*}
\left\{
\begin{aligned}
(-\Delta)^m \overline w &= |x|^\sigma && \text{ in } B_1, \\
\nabla^i \overline w \big|_{\partial B_1} &=0 && \text{ for } 0 \leq i \leq m-1.
\end{aligned}
\right.
\end{align*}
As $\sigma > -2m$, $\overline w \in C(\overline B_1)$, hence there exists $\lambda_0 > 0$ such that $1 \geq \lambda_0(1+\|\overline w\|_\infty)^p$. Let $\lambda \in (0, \lambda_0]$, readily $\overline w$ is a supersolution to \eqref{eqA}$_\lambda$. Consider the following iteration process:
\begin{align*}
\left\{
\begin{aligned}
(-\Delta)^m w_{k+1} &= \lambda |x|^\sigma (1 + w_k)^p && \text{ in } B_1, \\ \nabla^i w_{k+1} \big|_{\partial B_1} &=0 && \text{ for } 0 \leq i \leq m-1.
\end{aligned}
\right.
\end{align*}
Clearly, using the positivity of $\G$ and monotonicity of $t \mapsto (1+t)^p$ in $\R_+$, there holds
\begin{align*}
0 \leq w_k \leq w_{k+1} \leq \overline w, \quad \forall\; k \geq 0.
\end{align*}
It is easy to conclude that $$u_\lambda = \lim_{k\to +\infty}w_k$$ exists and $u_\lambda$ is a weak solution to \eqref{eqA}$_\lambda$. As $\overline w$ can be replaced by any solution of \eqref{eqA}$_\lambda$, we see that $u_\lambda$ is the minimal solution to \eqref{eqA}$_\lambda$. The uniqueness of the minimal solution and $\sigma > -2m$ guarantee that $u_\lambda \in C_{0,\rm rad} (\overline B_1)$. Here $C_{0,\rm rad} (\overline B_1)$ stands for the space of radial continuous functions in $\overline B_1$ with zero boundary value, equipped with the sup-norm $\|u \|_\infty$. 

Denote
\[
\Lambda = \big\{ \lambda > 0 : \eqref{eqA}_\lambda \text{ admits a positive solution in } C_{0,\rm rad} (\overline B_1) \big\}.
\]
As any solution of \eqref{eqA}$_\lambda$ is a supersolution to \eqref{eqA}$_\mu$ for $\mu \in (0, \lambda)$, $\Lambda$ is clearly an interval. We claim now $$\lambda^* = \sup \Lambda < +\infty.$$ Let $\Phi_{1,\sigma}$ be the first eigenfunction for the following eigenvalue problem
\begin{equation*}\label{eqEVP}
\left\{
\begin{aligned}
(-\Delta)^m u &= \lambda_{1,\sigma} |x|^\sigma u && \text{ in }\; B_1,\\
\nabla^i u \big|_{\partial B_1} &=0 && \text{ for }\; 0 \leq i \leq m-1.
\end{aligned}
\right.
\end{equation*}
It is not hard to see that $\Phi_{1,\sigma}$ can be obtained via standard argument as $\sigma > -2m$. Indeed,
\[
0 < \lambda_{1,\sigma} = \inf_{u \in H_0^{m,2}(B_1)\backslash\{0\}} 
\frac
{\|\Delta^{m/2} u\|^2_{L^2(B_1) }}
{ \|u\|^2_{L^2(B_1, |x|^\sigma dx)}}
\]
is attained. Notice that the positivity of $\G$ also implies the strong maximum principle, so that the corresponding first eigenfunction $\Phi_{1,\sigma}$ can be chosen to be positive in $B_1$. Now let $\lambda \in \Lambda$ and $u$ be a solution to \eqref{eqA}$_\lambda$, there holds
\begin{align*}
\lambda_{1,\sigma} \int_{B_1} |x|^\sigma u \Phi_{1,\sigma} 
&= \int_{B_1} u (-\Delta)^m \Phi_{1,\sigma} \\
&=\lambda \int_{B_1} |x|^\sigma (1+u)^p \Phi_{1,\sigma} 
\geq\lambda p\int_{B_1} |x|^\sigma u \Phi_{1,\sigma}.
\end{align*}
Since $u\Phi_{1,\sigma} > 0$ in $B_1$, we arrive at $\lambda \leq \lambda_{1, \sigma}/p$, in other words, $\lambda^* \leq \lambda_{1, \sigma}/p < +\infty$ as claimed. 

\begin{proposition}\label{prop-ExistenceA-Small}
There exists $0<\lambda^*<+\infty$ such that 
\begin{itemize}
\item we have a minimal solution $u_\lambda \in C_{0,\rm rad} (\overline B_1)$ to equation \eqref{eqA}$_\lambda$ for any $0<\lambda < \lambda^*$, and for any $x \in B_1$, the mapping $\lambda \mapsto u_\lambda(x)$ is increasing in $(0, \lambda^*)$;
\item for $\lambda > \lambda^*$, \eqref{eqA}$_\lambda$ has no solution.
\end{itemize}
Furthermore, given $\mu \in (0, \lambda^*)$, there exists a universal constant $C_\mu$ such that for any $u \in C_{0,\rm rad} (\overline B_1)$ solution to equation \eqref{eqA}$_\lambda$ with $\lambda \geq \mu$, there holds
\[
u (x) \leq C_\mu |x|^{-\theta}
\]
for all $x \in B_{1/4} \backslash \{ 0 \}$.
\end{proposition}

\begin{proof}
We are only left with the uniform upper bound for any solution to \eqref{eqA}$_\lambda$ for any $\lambda \geq \mu$. As $n>2m$, the Green function $\G$ satisfies the following two-sided estimate: for any $x, y \in B_1$,
\begin{equation}\label{eqTwoSidedGreen}
|x-y|^{2m-n} \min \Big\{ 1, \Big( \frac{(1-|x|)(1-|y|)}{|x-y|^2} \Big)^m \Big\} \lesssim \G(x,y) \lesssim |x-y|^{2m-n};
\end{equation}
see \cite[equation (4.24)]{GGS10}. Let $u \in C_{0,\rm rad} (\overline B_1)$ be a solution to \eqref{eqA}$_\lambda$ and $x \in B_{1/4} \backslash \{ 0 \}$. Take any
\[
y \in B_{|x|/4} \big(\frac {3x}4 \big) \subset B_{|x|/2} (x) \cap B_{|x|} (0),
\]
There holds $|x|/2 \leq |y| \leq |x|$, so that $1-|y| \geq |x-y|$, $1-|x| \geq |x-y|$ and $|x-y| \leq |x|/2$. By \eqref{eqTwoSidedGreen} we arrive at
\[
\G(x,y) \gtrsim |x-y|^{2m-n} \gtrsim |x|^{2m-n}.
\]
Moreover, putting the above facts together, the solution $u$ can be estimated as follows
\begin{align*}
u(x) &\gtrsim \mu \int_{B_{|x|/4} (3x/4)} \frac{|y|^\sigma}{|x-y|^{n-2m}} u^p(y) dy
 \gtrsim \mu |x|^{2m-n+\sigma} \times \frac {|x|^n}{4^n} u^p(x) .
\end{align*}
Here we used the fact that $u$ is decreasing with respect to the radius, see \cite[Theorem 2]{GeYe}. The above inequality gives us the desired estimate.
\end{proof}

\subsection{Uniqueness of radial solutions to (\ref{eqA})$_\lambda$ for $\lambda > 0$ small}
This subsection is devoted to show the uniqueness of solution to \eqref{eqA}$_\lambda$ for small $\lambda > 0$. We will use Schaaf's idea; see \cite{Sch00}, based on the Pohozaev's identity and supercritical exponent $p$.

\begin{lemma}\label{lem-Pohozaev}
Assume that $\Sigma \subset \R^n$ is a bounded, smooth domain and $f \in C^1(\overline \Sigma \times \R)$. Let $u$ be a $C^{2m}$-solution to
\[
(-\Delta)^m u = f(x,u) \quad \text{ in } \Sigma \subset \R^n.
\]
Denote by $\nu$ the unit outside normal vector on $\p\Sigma$ and 
\[
F(x,u) := \int_0^u f(x,t)dt.
\]
Then one has the following identities
\begin{equation}\label{eqPohozaev1}
\begin{aligned}
n \int_{\Sigma} F(x,u) & + \int_{\Sigma} x \cdot \nabla_x F(x,u) - \int_{\partial\Sigma} (x\cdot\nu) F(x, u) \\
= & \frac{n-2m}2 \int_{\Sigma} |\Delta^{m/2} u|^2 
+ \frac 12 \int_{\partial \Sigma} A(u, u) + \frac{1}{2}\int_{\partial\Sigma} T_m(x, u)
\end{aligned}
\end{equation}
and
\begin{equation}\label{eqPohozaev2}
\begin{aligned}
\int_{\Sigma} |\Delta^{m/2} u|^2 
 = \int_{\Sigma} u f(x, u) + \frac 12 \int_{\partial \Sigma} B(u, u) .
\end{aligned}
\end{equation}
Here the boundary terms have the form
\begin{align*}
A (u,v) 
= \sum_{j=1, j \ne m}^{2m-1} \overline l_j (x, \nabla^j u, \nabla^{2m-j} v) 
+\sum_{j=0}^{2m-1} \widetilde l_j (\nabla^j u, \nabla^{2m-j-1} v)
\end{align*}
and
\begin{align*}
B (u,v) 
= \sum_{j=0}^{2m-1} \widehat l_j (\nabla^j u, \nabla^{2m-1-j} v),
\end{align*}
where $\overline l_j $ is trilinear in $(x, u, v)$, $\widetilde l_j$, and $\widehat l_j$ are bilinear in $(u, v)$, and 
$$T_m(x, u) = \left\{\begin{array}{ll} \big(x\cdot\Delta^{m/2}u\big)\big(\nu\cdot\Delta^{m/2}u\big) & \;\; \mbox{if $m$ is odd},\\
\displaystyle \Delta^{m/2}u \times \sum_{1\leq i, j \leq n} x_i\nu_j \partial_{ij}\big(\Delta^{m/2 - 1}u\big) & \;\; \mbox{if $m$ is even}. 
\end{array}
\right.$$
\end{lemma}

To prove \eqref{eqPohozaev1} and \eqref{eqPohozaev2}, it is routine to use $x \cdot \nabla u$ and $u$ as test functions over $\Sigma$. Such computations are straightforward but tedious; however, for completeness, we provide a sketch of proof in Appendix \ref{apd-Pohozaev}. It is important to note that the boundary terms $A$ and $B$ in \eqref{eqPohozaev1} and \eqref{eqPohozaev2} do not depend on $f$. 

Now we are in position to prove the uniqueness result for radial solutions to \eqref{eqA}$_\lambda$ for small $\lambda > 0$. It is worth noting that the condition $p> \psnma$ is a crucial argument. 

\begin{lemma} \label{lem-Uniqueness}
Let $p> \psnma$. Then, there exists $\lambda_*>0$ such that for every $\lambda \in (0, \lambda_*]$, the equation \eqref{eqA}$_\lambda$ has a unique solution in $C_{0,\rm rad} (\overline B_1)$; hence coincides the minimal solution $u_\lambda $.
\end{lemma}

\begin{proof}
Let $\lambda \in \Lambda$ and $u_\lambda \in C_{0,\rm rad} (\overline B_1)$ be the minimal solution to \eqref{eqA}$_\lambda$, thanks to Proposition \ref{prop-ExistenceA-Small}. Let $v \in C_{0,\rm rad} (\overline B_1)$ be another solution to \eqref{eqA}$_\lambda$, there hold $w = v-u_\lambda \geq 0$ in $B_1$ and $w$ is a solution to
\begin{equation}
\label{new7}
\left\{
\begin{aligned}
(-\Delta)^m w &= \lambda |x|^\sigma f_\lambda (x, w) & & \text{ in } B_1,\\
\nabla^i w \big|_{\partial B_1} &=0 & & \text{ for } 0 \leq i \leq m-1,
\end{aligned}
\right.
\end{equation}
with 
\[
f_\lambda (x, w) =(1+ w + u_\lambda)^p - (1+u_\lambda)^p \geq 0.
\]
Our aim is to show that $w \equiv 0$ if $\lambda$ is small, hence concluding the claimed uniqueness for \eqref{eqA}$_\lambda$. Denote
\[
F_\lambda (x, w) = \int_0^w f_\lambda (x,t) dt.
\]
In the sequel, we apply Lemma \ref{lem-Pohozaev} with the solution $w$ of \eqref{new7} and $\Sigma = B_1 \backslash \overline B_\epsilon$, $\epsilon \in (0, 1)$. Thanks to the Dirichlet boundary conditions, there hold $F(x, w) = A(w, w) = B(w, w) = 0$ on $\p B_1$. Moreover, since $w$ is a radial function, we can check readily that for any $m$ and $r > 0$, there holds
\[
T_m(x, u) = (x\cdot \nu) |\Delta^{m/2} w|^2 \quad \mbox{on }\; \partial B_r.
\] 
In particular, $T_m(x, u) \geq 0$ on $\partial B_1$. Hence, from \eqref{eqPohozaev1} we have
\begin{align*}
 \lambda \int_{B_1 \backslash B_\epsilon} & \Big[ |x|^\sigma F_\lambda(x, w) + \frac 1n x \cdot \nabla_x \big( |x|^\sigma F_\lambda(x, w) \big) \Big] + \frac{\lambda \e^{\sigma + 1}}{n}\int_{\partial B_\e} F_\lambda(x, w) \\
 \geq & \; \frac{n-2m}{2n} \int_{B_1 \backslash B_\epsilon} |\Delta^{m/2} w|^2 - \frac{\e}{2n} \int_{\partial B_\e} |\Delta^{m/2} w|^2 
 + \frac 1{2n} \int_{\partial B_\epsilon} A(w, w) .
\end{align*}
From \eqref{eqPohozaev2}, we obtain, for any $\alpha \in \R$,
\[
\alpha\int_{B_1 \backslash B_\epsilon} |\Delta^{m/2} w|^2 
 = \alpha \lambda \int_{B_1 \backslash B_\epsilon} |x|^\sigma w f_\lambda (x, w) + \frac \alpha 2 \int_{\partial B_\epsilon} B(w, w).
\]
Combining these two estimates together, we arrive at
\begin{align}
\label{new8}
\begin{split}
& \Big( \frac{n-2m}{2n} - \alpha \Big) \int_{B_1 \backslash B_\epsilon} |\Delta^{m/2} w|^2 
 + J_\epsilon 
 \\
\leq & \, \lambda \int_{B_1 \backslash B_\epsilon} \Big[ |x|^\sigma F_\lambda(x, w) + \frac 1n x \cdot \nabla_x \big( |x|^\sigma F_\lambda(x, w) \big) - \alpha |x|^\sigma u f_\lambda (x, w) \Big] \\
= & \, \lambda \int_{B_1 \backslash B_\epsilon} |x|^\sigma \Big[ \big( 1 + \frac \sigma n \big) F_\lambda(x, w) - \alpha u f_\lambda (x, w) + \frac 1n x \cdot \nabla_x F_\lambda(x, w) \Big],
\end{split}
\end{align}
where 
\begin{align*}
J_\epsilon & = - \frac{\e}{2n} \int_{\partial B_\e} |\Delta^{m/2} w|^2 - \frac{\lambda \e^{\sigma + 1}}{n} \int_{\partial B_\e} F_\lambda(x, w) \\
& \quad + \frac 1{2n} \int_{\partial B_\epsilon} A(w, w) - \frac \alpha 2 \int_{\partial B_\epsilon} B(w,w).
\end{align*}
We will estimate $J_\epsilon $ and $\nabla_x F_\lambda(x, w) $ appearing in \eqref{new8}. For $\nabla_x F_\lambda(x, w)$, we note that
\begin{equation}\label{Fnearzero}
\begin{aligned}
F_\lambda(x, w) = w \int_0^1 f_\lambda (x,sw) ds
& = w \int_0^1 \big[ (1+ sw + u_\lambda)^p - (1+u_\lambda)^p \big] ds\\
&= p w^2 \int_0^1 s \int_0^1 (1+u_\lambda + \tau sw)^{p-1} d\tau ds.
\end{aligned}
\end{equation}
Therefore
\[
\nabla_x F_\lambda(x, w) = \Big[ p(p-1)w^2 \int_0^1 s\int_0^1 (1+u_\lambda + \tau s w)^{p-2} d\tau ds \Big] \nabla u_\lambda (x) .
\]
Using \cite[Theorem 2]{GeYe}, $u_\lambda$ is decreasing with respect to the radius, namely $x \cdot \nabla u_\lambda \leq 0$, so
\begin{align}
\label{new9}
x \cdot \nabla_x F_\lambda(x, w) \leq 0.
\end{align}
Now we estimate $J_\epsilon $. As $v, u_\lambda \in C_{0, \rm rad}(\overline B_1)$ and $\sigma > -2m$, the regularity theory ensures that $f_\lambda \in C^{0, \gamma}(\overline B_1)$ for some $\gamma > 0$. The scaling argument and the interior estimate, see \cite[Theorem 2.19]{GGS10}, applied to \eqref{new7} then imply
\begin{align*}
|\nabla^j w(x)| \leq C\left(|x|^{2m+\sigma - j} + 1\right), \quad \forall\; 1 \leq j \leq 2m, \; |x| \leq \frac{1}{2}. 
\end{align*}
Hence
\begin{align*}
|A (w,w) | + |B(w,w)| + \e|\Delta^{m/2} w|^2\leq C\big[\epsilon^{2(2m+\sigma) + 1 - 2m} + 1\big]\quad \mbox{on }\; \p B_\e.
\end{align*}
For the term involving $F_\lambda$, as $F_\lambda$ is bounded in a neighborhood of the origin, we get
\[
\e^{\sigma + 1} F_\lambda(x, w) \leq C \e^{\sigma + 1} \quad \mbox{on }\; \p B_\e.
\]
Finally, as $n > 2m > -\sigma$, there holds
\begin{align}
\label{new10}
\lim_{\e\to 0^+} J_\epsilon = 0.
\end{align}
Keep in mind that $F_\lambda \in L^1(B_1)$, $wf_\lambda \in L^1(B_1)$, and $\Delta^{m/2} w \in L^2(B_1)$. Putting \eqref{new8}, \eqref{new9}, and \eqref{new10} together and sending $\e\to 0^+$, we conclude that
\begin{align*}
 \Big( \frac{n-2m}{2n} - \alpha \Big) & \int_{B_1} |\Delta^{m/2} w|^2 \leq \lambda \int_{B_1 } |x|^\sigma \Big[ \big( 1 + \frac \sigma n \big) F_\lambda(x, w) - \alpha w f_\lambda (x, w) \Big] .
\end{align*}
From now on, we consider $\lambda \leq \lambda^*/2$, where $\lambda^*$ is given in Proposition \ref{prop-ExistenceA-Small}. By direct computation, we get
\[
F_\lambda (x, t) = \frac 1{p+1} \big[(1+ t + u_\lambda)^{p+1} - (1+ u_\lambda)^{p+1} \big] - t (1+u_\lambda)^p.
\]
As $0 \leq u_\lambda \leq \|u_{\lambda^*/2}\|_\infty$, we claim 
\[
\lim_{t \to +\infty} \frac{F_\lambda (x,t)}{t f_\lambda (x,t)} 
=\frac 1{p+1}\lim_{t \to +\infty} \frac{ (1+ t + u_\lambda)^{p+1} - (1+ u_\lambda)^{p+1}}{t \big[(1+ t + u_\lambda)^p - (1+ u_\lambda)^p \big]} 
= \frac 1{p+1}
\]
uniformly in $B_1$ and in $\lambda \leq \lambda^*/2$. Thus, combining with \eqref{Fnearzero}, for any $\delta >0$, there is $M_\delta >0$ such that
\[
F_\lambda (x,t) \leq \frac {1+\delta}{p+1} t f_\lambda (x, t) + M_\delta t^2
\]
for all $(x, t, \lambda) \in B_1 \times \R_+\times (0, \lambda^*/2]$. Then we choose $\alpha, \delta > 0$ satisfying
\[
\big( 1 + \frac \sigma n \big) \frac {1+\delta}{p+1} = \alpha <\frac{n-2m}{2n}.
\]
This can be done because
\[
\big( 1 + \frac \sigma n \big) \frac 1{p+1} <\frac{n-2m}{2n} \quad \iff \quad p+1> \frac{2(n+\sigma)}{n-2m}.
\]
With these choices, we have just shown that for $\lambda \leq \lambda^*/2$,
\begin{align}
\label{new11}
 \Big( \frac{n-2m}{2n} - \alpha \Big) \int_{B_1} |\Delta^{m/2} w|^2 
 & \leq \lambda \big( 1 + \frac \sigma n \big) M_\delta \int_{B_1} |x|^\sigma w^2 
\end{align}
Making use of the H\"older and Hardy--Sobolev inequalities, as $\sigma > -n$, there holds
\begin{align}\label{new111}
\begin{split}
\int_{B_1} |x|^\sigma w^2 & \leq \Big( \int_{B_1} |x|^\sigma |w|^\frac{2(n+\sigma)}{n-2m} dx\Big)^\frac{n-2m}{n+\sigma} \Big( \int_{B_1 } |x|^\sigma \Big)^\frac{\sigma+2m}{n+\sigma}\\
& \lesssim \int_{B_1} |\Delta^{m/2} w|^2.
\end{split}
\end{align}
Putting \eqref{new111} into \eqref{new11}, we obtain $\|\Delta^{m/2} w\|_{L^2(B_1)} = 0$ if $\lambda > 0$ is small enough. Coming back to inequality \eqref{new111}, together with the continuity of $w$, there holds $w \equiv 0$ in $B_1$, namely $u_\lambda$ is the unique solution in $C_{0, \rm rad}(\overline B_1)$ for $\lambda > 0$ small.
\end{proof}

\subsection{Existence of classical solutions to (\ref{eqMAIN})$_\s$}

Here we prove finally the existence of a classical solution to \eqref{eqMAIN}$_\s$ using solutions to the auxiliary problem \eqref{eqA}$_\lambda$.

\begin{lemma}\label{lem-Bifurcation}
There exists a sequence of $ (\lambda_k, u^{\lambda_k})$ in $(0, \lambda^*]\times C_{0,\rm rad} (\overline B_1)$ with $u^{\lambda_k}$ a solution to \eqref{eqA}$_{\lambda_k}$ such that
\[
\lim_{k \to +\infty} \lambda_k = \lambda_\infty > 0 \quad \mbox{and}\quad \lim_{k \to +\infty} \|u^{\lambda_k}\|_{\infty} = +\infty.
\]
\end{lemma}

\begin{proof}
Consider $\F : C_{0,\rm rad} (\overline B_1) \to C_{0,\rm rad} (\overline B_1)$ defined as follows: for $u \in C_{0,\rm rad} (\overline B_1)$, let $v = \F(u)$ be the unique solution to 
\begin{align*}
\left\{
\begin{aligned}
(-\Delta)^m v & = |x|^\sigma (1 + |u|)^p && \mbox{in }\; B_1, \\
\nabla^i v \big|_{\p B_1} &= 0 && \mbox{for }\; 0 \leq i \leq m-1.
\end{aligned}
\right.
\end{align*}
As $\sigma > -2m > -n$, by regularity theory and Sobolev embedding, we see readily that $\F$ is compact. For each $\lambda \in (0, \lambda^*)$, any solution $u^\lambda$ to \eqref{eqA}$_\lambda$ actually solves
\begin{align}
\label{new12}
u^\lambda = \lambda\F (u^\lambda).
\end{align}
It follows from \cite[Theorem 6.2]{Rab73} that the set of pairs $(\lambda, u^\lambda)$ satisfying \eqref{new12} is unbounded in $\R_+ \times C_{0,\rm rad} (\overline B_1)$. By Proposition \ref{prop-ExistenceA-Small}, we can extract an unbounded sequence $(\lambda_k, u^{\lambda_k}) \in (0, \lambda^*) \times C_{0,\rm rad} (\overline B_1)$. Up to a subsequence, we can assume that
\[
\lim_{k \to +\infty} \lambda_k = \lambda_\infty \in [0, \lambda^*] \quad \text{and} \quad \lim_{k \to +\infty} \|u^{\lambda_k}\|_\infty = +\infty.
\] 
Moreover, Lemma \ref{lem-Uniqueness} combined with the unboundedness of $\|u^{\lambda_k}\|_\infty$ means $\lambda_\infty > 0$.
\end{proof}

We are now ready to prove the existence of a classical solution to \eqref{eqMAIN}$_\sigma$ in $\R^n$. 

\begin{proof}[Proof of Theorem \ref{thm-classical-supercritical}]
Let $(\lambda_k, u^{\lambda_k})$ be the sequence provided by Lemma \ref{lem-Bifurcation}, as $u^{\lambda_k}$ is radially symmetric and decreasing with respect to the radius, we have
\[
u^{\lambda_k} (0)=\max_{\overline B_1} u^{\lambda_k}(x) \to +\infty.
\]
Set
\[
v_k(x) = \frac{u^{\lambda_k}(r_k x)}{u^{\lambda_k}(0)}
\]
with $r_k > 0$ satisfying
\[
\lambda_k r_k^{2m+\sigma} \big(u^{\lambda_k} (0) \big)^{p-1}= 1.
\]
Clearly, $r_k \to 0$ as $k \to +\infty$. It is easy to check that $v_k$ satisfies in $B_{1/r_k}$
\begin{equation*}\label{eqVk}
\begin{aligned}
(-\Delta)^m v_k (x) &= \frac{r_k^{2m}}{u^{\lambda_k}(0)} \lambda_k |r_kx|^\sigma \Big( u^{\lambda_k}(r_k x) + 1\Big)^p\\
&=|x|^\sigma \Big(v_k (x) + \frac 1{u^{\lambda_k}(0)}\Big)^p =: h_k (x).
\end{aligned}
\end{equation*}
Moreover, there hold $0 \leq v_k \leq 1$, $v_k(0)=1$, and $v_k$ is radially symmetric and decreasing with respect to the radius. 

Now let $R>0$ be arbitrary but fixed. As $\sigma > -2m$ and $|h_k(x)| \leq 2^p |x|^\sigma$ in $B_{1/r_k}$, then $h_k \in L^q(B_{2R})$ for $k$ large enough and some $q > n/(2m)$. Applying the $L^q$-theory to equation $(-\Delta)^m v_k = h_k$, see \cite[Corollary 2.21]{GGS10}, we know that
$v_k$ are bounded in $W^{2m,q}(B_R) \hookrightarrow C^{0, \gamma}(B_R)$
for some $\gamma \in (0,1)$. Therefore, up to a subsequence, there exists a function $v$ such that $v_k \to v$ locally uniformly in $\R^n$ and $v \in C_{\rm rad}(\R^n)$. Hence
\[
\mbox{$v$ is non-increasing with the radius, }\; v(0)=1, \;\; 0 \leq v(x) \leq 1.
\] 
Readily, $v$ is a continuous, distributional solution to \eqref{eqMAIN}$_\sigma$ in $\R^n$. The regularity theory implies that $v$ is a classical solution. 
\end{proof}

\begin{remark}
Using Proposition \ref{prop-IntegralEquation}, the limiting function $v \in C_{\rm rad}(\R^n)$ satisfies the integral equation \eqref{eqIntegralEquation}, which leads to the decay estimate $|v(x)| \leq C|x|^{-\theta}$ at infinity. Here we give a direct proof by Proposition \ref{prop-ExistenceA-Small}. Indeed, as $\lambda_\infty > 0$, we have
\begin{align*}
v_k (x) \leq u^{\lambda_k}(r_k x) \lesssim r_k^{-\theta} |x|^{-\theta} \lesssim \lambda_k^\frac{1}{p-1} |x|^{-\theta} 
\leq C |x|^{-\theta} 
\end{align*}
for any $x \in B_{1/(4r_k)}\backslash \{ 0 \}$. So we get $v(x) \lesssim |x|^{-\theta}$ in $\R^n \backslash \{ 0 \}$. 
\end{remark}

A quick consequence of Theorem \ref{thm-classical-supercritical} is the existence of a fast-decay punctured solution to \eqref{eqMAIN}$_\sigma$ in $\R^n \backslash \{ 0 \}$, which is different from the slow decay punctured solution $C_0 |x|^{-\theta}$. Given $n > 2m > -\sigma$ and $\pcnma < p < \psnma$. Consider \eqref{eqMAIN}$_{\widetilde\sigma}$ with $\widetilde \sigma = (n-2m)p - (n+2m+\sigma)$ and $p$. We check easily that $\widetilde\sigma > -2m$ and $p > \ps (m, \widetilde \sigma)$. Using the proof of Theorem \ref{thm-classical-supercritical}, there exists a radial classical solution $\widetilde u$ to \eqref{eqMAIN}$_{\widetilde\sigma}$. Clearly the Kelvin transform of $\widetilde u$, namely
\[
u(x) = |x|^{2m -n} \widetilde u \big(\frac x{|x|^2} \big),
\]
is a fast-decay punctured solution to \eqref{eqMAIN}$_\sigma$. Thus, we have just shown the following.

\begin{corollary}\label{cor-punctured-subcritical}
Let $n > 2m > -\s$ and $\pcnma < p < \psnma$. Then the equation \eqref{eqMAIN}$_\sigma$ admits a radial, fast-decay, punctured solution $u$ such that
\[
u(x) \sim 
\left\{
\begin{aligned}
& |x|^{-\theta} & & \text{ as } |x| \to 0,\\
& |x|^{2m-n} & & \text{ as } |x| \to +\infty.
\end{aligned}
\right.
\]
\end{corollary}

In the case $m=2$, Corollary \ref{cor-punctured-subcritical} is already known; see \cite[Theorem 2.2]{HS19}.


\section{Further remarks and some open questions}

From the discussion in this paper, we see that the polyharmonic Hardy-H\'enon equation \eqref{eqMAIN}$_\sigma$ is much more complex than the Laplacian case since many conclusions for $m = 1$ are no longer valid for $m \geq 2$. 

In the present work, we mainly studied the case: $n \geq 2m$, $\sigma > -2m$, and $p > 1$. Under these conditions, what we obtained are the following:

\begin{itemize}
\item a classical solution to \eqref{eqMAIN}$_\sigma$ exists \textbf{if and only if} $p \geq \psnma$;
\item a distributional solution to \eqref{eqMAIN}$_\sigma$ exists \textbf{if and only if} $p > \pcnma$;
\item a punctured solution to \eqref{eqMAIN}$_\sigma$ exists \textbf{if} $p > \pcnma$.
\end{itemize}

There is no ``and only if'' for punctured solutions to \eqref{eqMAIN}$_\sigma$ because by the examples after the proof of Proposition \ref{prop-Strong->Distribution}, we know that for any $m \geq 2$, there exist $n > 2m$ and $1 < p \leq \pcnma$ such that punctured solutions exist, and these solutions are not distributional ones. Moreover, recall that the inequality \eqref{new3} is a sufficient condition for the existence of punctured solutions. Thus, a natural question is to know if the condition \eqref{new3} is also necessary.

\noindent \textbf{Question 1}: Does a punctured solution to \eqref{eqMAIN}$_\sigma$ exist only if \eqref{new3} is satisfied? If the general answer is negative, is that true at least for $n \geq 2m$, $\sigma > -2m$, and $p > 1$?

For the existence of classical solutions to \eqref{eqMAIN}$_\sigma$, the situation seems very open for $n \leq 2m$ or $\sigma < -2m$. In Remark \ref{rem:j1}, we see some examples of classical solutions, which are not distributional ones with $n < 2m = 6$. There are many other examples, here are some ones for the biharmonic case. Let $m = 2$ and $\sigma < -4$, then 
\[
\begin{aligned}
&\text{if } \;  \;  n = 2, \; \theta < 0, \; \theta \ne -2;\\
&\text{or } \;  \;  n = 3, \; \theta \in (-\infty, -2)\cup(-1, 0);\\
&\text{or } \;  \;  n \geq 4, \;  \theta < -2,
\end{aligned}
\]
a classical solution of \eqref{eqMAIN}$_\sigma$ exists in the form $C|x|^{-\theta}$, because \eqref{new3} is satisfied. A striking observation is that any of the above examples of classical solution does not satisfy the SPH property. 

\noindent \textbf{Question 2}: Let $n < 2m$ or $\sigma < -2m$, for which $p > 1$ there exist classical solutions? Moreover, can we have classical solutions satisfying the strong SPH property?

There are very few results for the existence or non-existence of solutions to the polyharmonic equation \eqref{eqMAIN}$_\sigma$ with $0 < p < 1$ and $m \geq 2$. If the case $\sigma < 0$ could yield more difficulty in general, we can ask 

\noindent \textbf{Question 3}: Let $\sigma > 0$ and $m\geq 2$, for which $p \in (0, 1)$ a classical solution to \eqref{eqMAIN}$_\sigma$ exists?

Choosing suitable $p < 1$ and $\sigma$, the examples after Question 1 provide us some classical solutions to \eqref{eqMAIN}$_\sigma$. Once again, the situation is totally different from the second order case. In fact, Dai and Qin proved that no classical solution to $-\Delta u = |x|^\sigma u^p$ exists for any $\sigma \in \R$ and $p \in (0, 1]$, see \cite[Theorem 1.1]{DQ20}. For the case $m \geq 2$, the results in \cite{NNPY18} could mean that the answer to Question 3 will depend on the parity of $m$.

At last, by Theorem \ref{thm-GeneralExistence-Distributional}, the condition $n - 2m - \theta > 0$ is necessary to have a distributional solution, but except the case $\sigma > -2m$, we don't know if it is always necessary. Hence we can ask

\noindent \textbf{Question 4}: Is there always a distributional solution to \eqref{eqMAIN}$_\sigma$ when $n - 2m - \theta > 0$ and $\sigma \leq -2m$?

To conclude, the existence and non-existence problems to the polyharmonic equation \eqref{eqMAIN}$_\sigma$ keeps a lot of secrets when $m \geq 2$ and $\sigma \ne 0$. Apparently, we are very far away from a complete picture for the case $m = 1$ or for the case $\sigma = 0$; see \cite{NNPY18} where a complete picture is known for classical solutions. Possible answers could depend on many factors including the sign of $n- 2m$, the sign of $\sigma + 2m$, the sign of $p-1$, and the required regularity of solutions.

\section*{Acknowledgments}

This work was initiated when QAN was visiting the Center for PDEs at the East China Normal University in 2019. He would like to thank them for hospitality and financial support. Thanks also go to Quoc Hung Phan for useful discussion on the work \cite{PhanSouplet}. QAN is supported by the Tosio Kato Fellowship awarded in 2018. DY is partially supported by Science and Technology Commission of Shanghai Municipality (STCSM) under grant No. 18dz2271000.


\appendix


\section{Proof of Lemma \ref{lem-Pohozaev}}
\label{apd-Pohozaev}

For completeness, we provide here a proof of Lemma \ref{lem-Pohozaev}. First, we need the following two identities. Let $u, v \in C^{2m}(\overline \Sigma) \cap C^{2m-1}(\overline \Sigma)$,
\begin{equation}\label{eq-Identity2}
\begin{aligned}
\int_{\Sigma} & \big[ v (-\Delta)^m u + u (-\Delta)^m v \big] 
= -\int_{\partial \Sigma} B_m (u, v) + 2 \int_{\Sigma} \Delta^{m/2} u \Delta^{m/2} v 
\end{aligned}
\end{equation}
and
\begin{equation}\label{eq-Identity1}
\begin{aligned}
\int_{\Sigma} \big[ (x \cdot \nabla v)& (-\Delta)^m u + (x \cdot \nabla u) (-\Delta)^m v \big] \\
&=- \int_{\partial \Sigma} C_m (u, v) - \frac{n-2m}2 \int_{\Sigma} \big[ v (-\Delta)^m u + u (-\Delta)^m v \big] ,
\end{aligned}
\end{equation}
where the boundary term $B_m (u,v)$ is that in Lemma \ref{lem-Pohozaev}, that is
\begin{align*}
B_m (u,v) 
= \sum_{j=0}^{2m-1} \widehat l_{m,j} (\nabla^j u, \nabla^{2m-1-j} v),
\end{align*}
and the boundary term $C_m (u,v)$ is of the form
\begin{align*}
C_m (u,v) 
= \sum_{j=1}^{2m-1} \overline l_{m,j} (x, \nabla^j u, \nabla^{2m-j} v) 
+\sum_{j=0}^{2m-1} \breve l_{m,j} (\nabla^j u, \nabla^{2m-j-1} v).
\end{align*}
Here $\overline l_{m,j}$ is trilinear in $(x, u, v)$, $\breve l_{m,j}$, and $\widehat l_{m,j}$ are bilinear in $(u, v)$. The identities \eqref{eq-Identity2}, \eqref{eq-Identity1} can be proved directly using integration by parts, see for example \cite[Proposition 3.3]{GPY17}. 

The only thing we need to verify is the precise formula for the term $\overline l_{m,m}$. In fact, this term comes from the integration by parts for
\begin{align*}
\int_\Sigma \Delta^{(m-1)/2}(x\cdot \nabla v)\Delta^{(m+1)/2}u. 
\end{align*}
Let $m = 2k + 1$, as $\Delta^k(x\cdot \nabla v) = 2k\Delta^k v + x\cdot\nabla(\Delta^k v)$, we obtain
\begin{align*}
\int_\Sigma \Delta^k(x\cdot \nabla v)\Delta^{k+1}u = &\; -\int_\Sigma \Delta^{m/2}(x\cdot \nabla v)\Delta^{m/2}u \\
& + 2k\int_{\partial \Sigma}\Delta^k v (\nu\cdot\nabla \Delta^k u) + \int_{\partial \Sigma} (x\cdot\nabla\Delta^k v)(\nu\cdot\nabla \Delta^k u).
\end{align*}
We can notice that the first boundary term belongs to $\breve l_{m, m}$, while the last one yields that 
$$2\overline l_{m,m}(x, \nabla^m u, \nabla^m v) = (x\cdot\nabla\Delta^k v)(\nu\cdot\nabla \Delta^k u) + (x\cdot\nabla\Delta^k v)(\nu\cdot\nabla \Delta^k u).$$
As $T_m(x, u) = \overline l_{m,m}(x, \nabla^m u, \nabla^m u)$, we are done for $m$ odd. The case for $m$ even is completely similar, so we omit the details.

Now we are ready to prove Lemma \ref{lem-Pohozaev}. Using $u$ as testing function to the equation, by \eqref{eq-Identity2}, we obtain
\[
 \int_{\Sigma} u f(x,u) = \int_{\Sigma} u (-\Delta)^m u 
= -\frac 12 \int_{\partial \Sigma} B_m (u, u) + \int_{\Sigma} |\Delta^{m/2} u|^2,
\]
namely \eqref{eqPohozaev2} holds. For \eqref{eqPohozaev1}, we use $x \cdot \nabla u$ as testing function. There holds then
\begin{align}\label{eqPohozaev0}
\begin{split}
\int_{\Sigma} (x \cdot \nabla u) (-\Delta)^m u & = \int_{\Sigma} (x \cdot \nabla u) f(x,u) \\
& = \int_{\Sigma} x\cdot\nabla\big[F(x,u)\big] - \sum_{i=1}^n\int_{\Sigma} x_i \int_0^u \frac{\partial f}{\partial x_i}(x,t) dt\\
& =- n \int_{\Sigma} F(x,u) - \int_{\Sigma} x \cdot \nabla_x F(x,u) + \int_{\partial\Sigma} (x\cdot\nu) F(x, u),
\end{split}
\end{align}
where we formally denote 
$$\nabla_x F(x,u) = \Big(\int_0^u \frac{\partial f}{\partial x_i}(x,t) dt\Big)_{1\leq i \leq n}.$$
Taking $v = u$ in \eqref{eq-Identity1} and combining with \eqref{eqPohozaev0}, we arrive at
\begin{align*}
n \int_{\Sigma} F(x,u) + \int_{\Sigma} x \cdot \nabla_x F(x,u) 
= & \; \frac{n-2m}2 \int_{\Sigma} u (-\Delta)^m u\\
& + \frac 12 \int_{\partial \Sigma} C_m (u, u) + \int_{\partial\Sigma} (x\cdot\nu) F(x, u).
\end{align*}
From this and \eqref{eqPohozaev2} we obtain readily \eqref{eqPohozaev1}. This completes the proof. \qed



\begin{thebibliography}{wwwww}

\bibitem[AGQ16]{AGQ16}
\textsc{S. Alarc\'on, J. Garc\'ia-Meli\'an, and A. Quaas}, Optimal Liouville theorem for supersolutions of elliptic equations with the laplacian, 
\textit{Ann. Sc. Norm. Super. Pisa Cl. Sci.} \textbf{168} (2016) 129--158.

\bibitem[ACD$\dagger$19]{ACDFGW}
\textsc{W. Ao, H. Chan, A. DelaTorre, M.A. Fontelos, M. del Mar Gonz\'alez, and J.C. Wei},
On higher-dimensional singularities for the fractional Yamabe problem: A nonlocal Mazzeo--Pacard program,
\textit{Duke Math. J.} \textbf{168} (2019) 3297--3411.

\bibitem[BC98]{BC98}
\textsc{H. Brezis and X. Cabr\'e},
Some simple nonlinear PDE's without solutions,
\textit{Boll. Unione Mat. Ital.} \textbf{1-B } (1998) 223--262.


\bibitem[CAM08]{CAM08}
\textsc{G. Caristi, L. D'Ambrosio, and E. Mitidieri},
Representation formulae for solutions to some classes of higher order systems and related {L}iouville theorems,
\textit{Milan J. Math.} \textbf{76} (2008) 27--67.
	
	
\bibitem[CDQ18]{CDQ18}
\textsc{W. Chen, W. Dai, and G. Qin}, 
Liouville type theorems, a priori estimates and existence of solutions for critical order Hardy--H\'enon equations in $\mathbb R^n$, 
arXiv:1808.06609v4.

\bibitem[CLi91]{CLi91}
\textsc{W. Chen and C. Li}, 
Classification of solutions of some nonlinear elliptic equations. 
\textit{Duke Math. J.} \textbf{63} (1991), 615-622. 

\bibitem[CLi13]{ChenLi13}
\textsc{W. Chen and C. Li}, 
Super polyharmonic property of solutions for PDE systems and its applications, 
\textit{Commun. Pure Appl. Anal.} \textbf{12} (2013) 2497--2514.

\bibitem[CL16]{CL16}
\textsc{T. Cheng and S. Liu},
A Liouville type theorem for higher order Hardy--H\'enon equation in $R^n$,
\textit{J. Math. Anal. Appl.} \textbf{444} (2016) 370--389.

\bibitem[DPQ18]{DaiPengQin-Liouville-v4}
\textsc{W. Dai, S. Peng, and G. Qin},
Liouville type theorems, a priori estimates and existence of solutions for non-critical higher order Lane--Emden--Hardy equations, arXiv:1808.10771.

\bibitem[DQ19]{DaiQin-Liouville-v6}
\textsc{W. Dai and G. Qin},	
Liouville type theorems for fractional and higher order H\'enon--Hardy type equations via the method of scaling spheres, arXiv:1810.02752v7.

\bibitem[DQ20]{DQ20}
\textsc{W. Dai and G. Qin},	
Liouville type theorems for Hardy--H\'enon equations with concave nonlinearities, \textit{Math. Nachr.} \textbf{293} (2020) 1084--1093.
	
\bibitem[DDG11]{DDG11}
\textsc{E.N. Dancer, Y. Du, and Z. Guo},
Finite Morse index solutions of an elliptic equation with supercritical exponent,
\textit{J. Differential Equations} \textbf{250} (2011) 3281--3310. 
	
\bibitem[DN17]{DN17}
\textsc{T.V. Duoc and Q.A. Ng\^o},
\textit{Exact growth at infinity for radial solutions to $\Delta^3 u +u^{-q} = 0$ in $\R^3$}, preprint, 2017.
		
\bibitem[FWX15]{FWX15}
\textsc{M. Fazly, J.C. Wei, and X. Xu},
A pointwise inequality for the fourth-order {L}ane--{E}mden equation,
\textit{Anal. PDE} \textbf{8} (2015) 1541--1563.


\bibitem[HS19]{HS19}
\textsc{A. Hyder and Y. Sire},
Singular solutions for the constant $Q$-curvature problem, arXiv: 1911.11891

\bibitem[GG06]{GG06}
\textsc{F. Gazzola and H.-C. Grunau},
Radial entire solutions for supercritical biharmonic equations,
\textit{Math. Ann.} \textbf{334} (2006) 905--936.

\bibitem[GGS10]{GGS10}
\textsc{F. Gazzola, H.-C. Grunau, and G. Sweers},
Polyharmonic boundary value problems, \textit{Lecture Notes in Mathematics} 1991, Springer-Verlag, Berlin, 2010.

\bibitem[GY02]{GeYe}
\textsc{Y. Ge and D. Ye},
Monotonicity of radially symmetric supersolutions for polyharmonic-type operators,
\textit{Differential Integral Equations} \textbf{15} (2002) 357--366.

\bibitem[GS81]{GS81}
\textsc{B. Gidas and J. Spruck},
Global and local behavior of positive solutions of nonlinear elliptic equations,
\textit{Comm. Pure Appl. Math.} \textbf{34} (1981) 525--598.

\bibitem[GM99]{GM99}
\textsc{L. Grafakos and C. Morpurgo},
A Selberg integral formula and applications,
\textit{Pacific J. Math.} \textbf{191} (1999) 85--94. 

\bibitem[GHY18]{GHY18}
\textsc{Z. Guo, X. Huang, and D. Ye},
Existence and nonexistence results for a weighted elliptic equation in exterior domains, 
\textit{Z. Angew. Math. Phys.} \textbf{71} (2020) 116. 

\bibitem[GPY17]{GPY17}
\textsc{Y. Guo, S. Peng, and S. Yan},
Local uniqueness and periodicity induced by concentration,
\textit{Proc. London Math. Soc.} \textbf{114} (2017) 1005--1043.

\bibitem[GW17]{GW17}
\textsc{Z. Guo and F. Wan},
Further study of a weighted elliptic equation,
\textit{Sci. China Math.} \textbf{60} (2017), 2391--2406. 

\bibitem[JL14]{JL14}
\textsc{E. Jannelli and A. Loiudice},
Critical polyharmonic problems with singular nonlinearities,
\textit{Nonlinear Anal.} \textbf{110} (2014) 77--96.

\bibitem[Lei13]{Lei13}
\textsc{Y. Lei},
Asymptotic properties of positive solutions of the Hardy-Sobolev type equations.,
\textit{J. Differential Equations} \textbf{254} (2013) 1774--1799.

\bibitem[LV16]{LV16}
\textsc{C. Li and J. Villavert},
Existence of positive solutions to semilinear elliptic systems with supercritical growth, \textit{Comm. Partial Differential Equations} \textbf{41} (2016) 1029--1039.

\bibitem[LL01]{LiebLoss}
\textsc{E.H. Lieb and M. Loss},
\textit{Analysis},
Graduate studies in Mathematics, {\bf 14}, American Mathematical Society, Providence, RI, 2001.


\bibitem[Lio85]{Lions-1985}
\textsc{P.L. Lions}, 
The concentration-compactness principle in the calculus of variations. The limit case, Part 2, 
\textit{Rev. Mat. Iberoam.} \textbf{1} (1985) 45--121.

\bibitem[LGZ06]{LGZ06b}
\textsc{J.Q. Liu, Y. Guo, and Y.J. Zhang},
Existence of positive entire solutions for polyharmonic equations and systems,
\textit{J. Partial Differential Equations} \textbf{19} (2006) 256--270.
	
\bibitem[Lin98]{Lin98}
\textsc{C.-S. Lin},
A classification of solutions of a conformally invariant fourth order equation in {${\mathbb R}^n$},
\textit{Comment. Math. Helv.} \textbf{73} (1998) 206--231.
	
\bibitem[MR03]{KR}
\textsc{P.J. McKenna and W. Reichel}, 
Radial solutions of singular nonlinear biharmonic equations and applications to conformal geometry, 
\textit{Electron. J. Differential Equations} \textbf{37} (2003) 1--13.
	
\bibitem[MP01]{MP01}
\textsc{E. Mitidieri and S.I. Pohozaev},
A priori estimates and the absence of solutions of nonlinear partial differential equations and inequalities,
\textit{Tr. Mat. Inst. Steklova} \textbf{234} (2001) 1--384.
		
\bibitem[Mit18]{Mit18}
\textsc{D. Mitrea}, 
Distributions, Partial Differential Equations, and Harmonic Analysis, Universitext, 2018.

\bibitem[NN$\dagger$18]{NNPY18}
\textsc{Q.A. Ng\^o, V.H. Nguyen, Q.H. Phan, and D. Ye},
Exhaustive existence and non-existence results for some prototype polyharmonic equations, arXiv:1802.05956, 2018.

\bibitem[Ni82]{Ni82}
\textsc{W.M. Ni},
On the elliptic equation $\Delta u+K(x)u^{(n+2)/(n-2)}=0$, its generalizations, and applications in geometry, 
\textit{Indiana Univ. Math. J.} \textbf{31} (1982) 493--529. 

\bibitem[Ni86]{Ni86}
\textsc{W.M. Ni},
Uniqueness, nonuniqueness and related questions of nonlinear elliptic and parabolic equations,
\textit{Nonlinear functional analysis and its applications, Part 2}, 229--241, Proc. Sympos. Pure Math., 45, Part 2, Amer. Math. Soc., Providence, RI, 1986.

\bibitem[PS12]{PhanSouplet}
\textsc{Q.H. Phan and P. Souplet},
Liouville-type theorems and bounds of solutions of Hardy--H\'enon equations,
\textit{J. Differential Equations} \textbf{252} (2012) 2544--2562.

\bibitem[Rab73]{Rab73}
\textsc{P. Rabinowitz}, 
Some aspects of nonlinear eigenvalue problems,
\textit{Rocky Mountain J. Math.} \textbf{3} (1973) 161--202.

\bibitem[RZ00]{RZ00}
\textsc{W. Reichel and H. Zou},
Non-existence results for semilinear cooperative elliptic systems via moving spheres,
\textit{J. Differential Equations} \textbf{161} (2000) 219--243.

\bibitem[Sch00]{Sch00}
\textsc{R. Schaaf}, 
Uniqueness for semilinear elliptic problems: supercritical growth and domain geometry, 
\textit{Adv. Differential Equations} \textbf{5} (2000) 1201--1220.

\bibitem[SZ96]{SZ96}
\textsc{S. Serrin and H. Zou},
Non-existence of positive solutions of Lane-Emden systems,
\textit{Differential Integral Equations} \textbf{9} (1996) 635--653.

\bibitem[Ste70]{Stein}
\textsc{E.M. Stein},
Singular integrals and differentiability properties of functions,
\textit{Princeton Mathematical Series}, No. 30 Princeton University Press, Princeton, N.J. 1970 xiv+290 pp. 

\bibitem[Vil14]{Vil14}
\textsc{J. Villavert},
Shooting with degree theory: Analysis of some weighted poly-harmonic systems,
\textit{J. Differential Equations} \textbf{257} (2014) 1148--1167.
	
\bibitem[WX99]{WX99}
\textsc{J. Wei and X. Xu}, 
Classification of solutions of higher order conformally invariant equations,
\textit{Math. Ann.} \textbf{313} (1999) 207--228.

\bibitem[Xu00]{Xu00}
\textsc{X. Xu},
Uniqueness theorem for the entire positive solutions of biharmonic equations in $\mathbf R^n$, \textit{Proc. Roy. Soc. Edinburgh Sect. A} \textbf{130} (2000) 651--670. 


\end{thebibliography}
\end{document}